\documentclass{article}%
\usepackage{amsmath}
\usepackage{graphicx}
\usepackage{amsfonts}
\usepackage{amssymb}
\usepackage{times}%
\providecommand{\U}[1]{\protect\rule{.1in}{.1in}}
\newtheorem{theorem}{Theorem}

\newtheorem{corollary}{Corollary}

\newtheorem{lemma}{Lemma}[section]

\newtheorem{proposition}{Proposition}

\newenvironment{proof}[1][Proof]{\textbf{#1.} }{\ \rule{1em}{1em}}
\numberwithin{equation}{section}
\begin{document}

\title{Instability of nonlinear dispersive solitary waves }
\author{Zhiwu Lin\\Mathematics Department\\University of Missouri\\Columbia, MO 65211 USA}
\date{}
\maketitle

\begin{abstract}
We consider linear instability of solitary waves of several classes of
dispersive long wave models. They include generalizations of KDV, BBM,
regularized Boussinesq equations, with general dispersive operators and
nonlinear terms. We obtain criteria for the existence of exponentially growing
solutions to the linearized problem. The novelty is that we dealt with models
with nonlocal dispersive terms, for which the spectra problem is out of reach
by the Evans function technique. For the proof, we reduce the linearized
problem to study a family of nonlocal operators, which are closely related to
properties of solitary waves. A continuation argument with a moving kernel
formula are used to find the instability criteria. Recently, these techniques
have also been extended to study instability of periodic waves and to the full
water wave problem.

\end{abstract}

\section{Introduction}

We consider the stability and instability of solitary wave solutions of
several classes of equations modeling weakly nonlinear, dispersive long waves.
More specifically, we establish criteria for the linear exponential
instability of solitary waves of BBM, KDV, and regularized Boussinesq type
equations. These equations respectively have the forms:

1. BBM type%

\begin{equation}
\partial_{t}u+\partial_{x}u+\partial_{x}f\left(  u\right)  +\partial
_{t}\mathcal{M}u=0; \label{BBM}%
\end{equation}

2. KDV type
\begin{equation}
\partial_{t}u+\partial_{x}f\left(  u\right)  -\partial_{x}\mathcal{M}u=0;
\label{kdv}%
\end{equation}

3. Regularized Boussinesq (RBou) type%
\begin{equation}
\partial_{t}^{2}u-\partial_{x}^{2}u-\partial_{x}^{2}f\left(  u\right)
+\partial_{t}^{2}\mathcal{M}u=0. \label{RBOU}%
\end{equation}
Here, the pseudo-differential operator $\mathcal{M}$ is defined as
\[
\left(  \mathcal{M}g\right)  \left(  k\right)  =\alpha\left(  k\right)
\hat{g}\left(  k\right)  ,
\]
where $\hat{g}$ is the Fourier transformation of $g$. Throughout this paper,
we assume: i) $f$ is $C^{1}$ with $f\left(  0\right)  =f^{\prime}\left(
0\right)  =0,\ $and $f\left(  u\right)  /u\rightarrow\infty.\ $ii)
$a\left\vert k\right\vert ^{m}\leq\alpha\left(  k\right)  \leq b\left\vert
k\right\vert ^{m}$ for large $k$, where $m\geq1$ and $a,b>0$. If $f\left(
u\right)  =u^{2}$ and $\mathcal{M}=-\partial_{x}^{2}$, the above equations
recover the original BBM (\cite{bbm72}), KDV (\cite{KDV}), and regularized
Boussinesq (\cite{whitham}) equations, which have been used to model the
unidirectional propagation of water waves of long wavelengths and small
amplitude. As explained in \cite{bbm72}, the nonlinear term $f\left(
u\right)  $ is related to nonlinear effects suffered by the waves being
modeled, while the form of the symbol $\alpha$ is related to dispersive and
possibly, dissipative effects. If $\alpha\left(  k\right)  $ is a polynomial
function of $k$, then $\mathcal{M}$ is a differential operator and in
particular is a local operator. On the other hand, in many situations in fluid
dynamics and mathematical physics, equations of the above type arise in which
$\alpha\left(  k\right)  $ is not a polynomial and hence the operator
$\mathcal{M}$ is nonlocal. Some examples include: Benjamin-Ono equation
(\cite{bno}), Smith equation (\cite{smith}) and intermediate long-wave
equation (\cite{ILW-eqn}), which are all of KDV type with $\alpha\left(
k\right)  =\left\vert k\right\vert ,\ \sqrt{1+k^{2}}-1$ and $k\coth\left(
kH\right)  -H^{-1}$ respectively.

Below we assume $\alpha\left(  k\right)  \geq0$, since the results and proofs
can be easily modified for cases of sign-changing symbols (see Section 5(b)).
Each of the equations (\ref{BBM})-(\ref{RBOU}) admits solitary-wave solutions
of the form $u\left(  x,t\right)  =u_{c}\left(  x-ct\right)  $ for
$c>1,c>0,c^{2}>1$ respectively, where $u_{c}\left(  x\right)  \rightarrow0$ as
$\left\vert x\right\vert \rightarrow\infty$. For example, the KDV solitary
wave solutions have the form (\cite{KDV})
\[
u_{c}\left(  x\right)  =3c\operatorname{sech}^{2}\left(  \sqrt{c}x/2\right)
\]
and for the Benjamin-Ono equation (\cite{bno})
\[
u_{c}\left(  x\right)  =\frac{4c}{1+c^{2}x^{2}}.
\]
For a broad class of symbols $\alpha\,$, the existence of solitary-wave
solutions has been established (\cite{benjamin73}, \cite{benjamin90}). For
many equations such as the classical KDV and BBM, the solitary waves are
positive, symmetric and single-humped. But the oscillatory solitary waves are
not uncommon (\cite{amt2}, \cite{abr-benjamin}), especially for the sign
changing $\alpha\left(  k\right)  $. In our study, we do not assume any
additional property of solitary waves, besides their decay at infinity. We
consider the linearized equations around solitary waves in the traveling frame
$\left(  x-ct,t\right)  $ and seek a growing mode solution of the form
$e^{\lambda t}u\left(  x\right)  $ with $\operatorname{Re}\lambda>0$. Define
the operator $\mathcal{L}_{0}$ by (\ref{L0-BBM}), (\ref{L0-KDV}), and
(\ref{L0-RBou}), and the momentum function $P\left(  c\right)  $ by
(\ref{P-BBM}), (\ref{P-KDV}), and (\ref{P-RBou}), for BBM, KDV and RBou type
equations respectively.

\begin{theorem}
\label{THM: main}For solitary waves $u_{c}\left(  x-ct\right)  \ $of equations
(\ref{BBM})-(\ref{RBOU}), we assume
\begin{equation}
\ker\mathcal{L}_{0}=\left\{  u_{cx}\right\}  . \label{assum-kernel}%
\end{equation}
Denote by $n^{-}\left(  \mathcal{L}_{0}\right)  $ the number (counting
multiplicity) of negative eigenvalues of the operators $\mathcal{L}_{0}$. Then
there exists a purely growing mode $e^{\lambda t}u\left(  x\right)  $ with
$\lambda>0,\ u\in H^{m}\left(  \mathbf{R}\right)  $ to the linearized
equations (\ref{L-bbm}), (\ref{L-KDV}) and (\ref{L-RBou}), if one of the
following two conditions is true:

(i) $n^{-}\left(  \mathcal{L}_{0}\right)  $ is even and $dP/dc>0.$

(ii) $n^{-}\left(  \mathcal{L}_{0}\right)  $ is odd and $dP/dc<0.$
\end{theorem}

Note that the operators $\mathcal{L}_{0}$ are obtained from the linearization
of equations satisfied by solitary waves, and $P\left(  c\right)  =Q\left(
u_{c}\right)  $ where $Q\left(  u\right)  $ is the momentum invariant due to
the translation symmetry of the evolution equations (\ref{BBM})-(\ref{RBOU}%
)$.$For example, for KDV type equitation, $Q\left(  u\right)  =\frac{1}{2}\int
u^{2}\ dx$. The assumption (\ref{assum-kernel}) can be proved for
$\mathcal{M}=-\partial_{x}^{2}\ $and for some nonlocal dispersive operators
(\cite{albert92}, \cite{albert-et-91}). It has the implication that the
solitary wave branch $u_{c}\left(  x\right)  $ is unique. More discussions
about the spectrum assumptions for $\mathcal{L}_{0}\ $can be found in Section 5(a).

Let us relate our results to the literature on stability and instability of
solitary waves. The first rigorous proof of stability of solitary waves is
obtained by Benjamin (\cite{benjamin72}), for the original KDV equation.
Benjamin's idea is to show that stable solitary waves are local energy
minimizers under the constraint of constant momentum. This idea was already
anticipated by Boussinesq (\cite{B0u-1872}) and has been extended to get
stability results for more general settings (\cite{albert-et-87},
\cite{albert-ccompact}, \cite{gss87}, \cite{wein-87}). In particular, it is
shown in \cite{bss87}, \cite{ss89} that for KDV and BBM type equations, the
solitary waves are orbitally stable in the energy norm if and only if
$dP/dc>0$, under the hypothesis%

\begin{equation}
\ker\mathcal{L}_{0}=\left\{  u_{cx}\right\}  ,\ \text{and }n^{-}\left(
\mathcal{L}_{0}\right)  =1.\ \label{kernel-1}%
\end{equation}
For power like nonlinear terms and dispersive operators with symbols
$\alpha\left(  k\right)  =\left\vert k\right\vert ^{\mu}$, the function
$P\left(  c\right)  $ can be computed by scaling and thus the more explicit
stability criteria is obtained (see \cite{bss87}, \cite{ss89}). The stability
criterion $dP/dc>0$ in \cite{bss87}, \cite{ss89} is by a straight application
of the abstract theory of \cite{gss87}, and this is also proved in
\cite{wein-87}. The instability proof of \cite{gss87} can not apply directly
to KDV and BBM cases. In \cite{bss87}, \cite{ss89}, the proof of \cite{gss87}
is modified to yield the instability criterion $dP/dc<0,\ $by estimating the
sublinear growth of the anti-derivative of the solution. A less technical way
of modification (introduced in \cite{lin-bubble}) is described in Appendix for
general settings. Applying Theorem \ref{THM: main} to the KDV and BBM cases
with $n^{-}\left(  \mathcal{L}_{0}\right)  =1$, we recover the instability
criterion $dP/dc<0\ $in \cite{bss87}, \cite{ss89}, and furthermore it helps to
clarify the mechanism of this instability by finding a non-oscillatory and
exponentially growing solution to the linearized problem. We note that the
nonlinear instability proved in \cite{bss87}, \cite{ss89} is in the energy
norm $H^{m/2}$ and there is no estimate of the time scale for the growth of
instability. The linear instability result might be the first step toward
proving a stronger nonlinear instability result in $L^{2}$ norm with the
exponential growth.

When $\mathcal{M}=-\partial_{x}^{2}$, Pego and Weinstein \cite{pw92-evans}
study the spectral problem for solitary waves of BBM, KDV and RBou equations
by the Evans function technique (\cite{agj-evans}, \cite{evans2}), and a
purely growing mode is shown to exist if $dP/dc<0$. Since for $\mathcal{M}%
=-\partial_{x}^{2}$, we have $\ker\mathcal{L}_{0}=\left\{  u_{cx}\right\}  $
and $n^{-}\left(  \mathcal{L}_{0}\right)  =1$ (see Section 5(a)), the result
of \cite{pw92-evans} is a special case of Theorem \ref{THM: main}. The novelty
of our result is to allow general dispersive operators $\mathcal{M}$,
particularly the nonlocal operators, for which the spectral problem can not be
studied via the Evans functions. Comparison with the Evans functions are
discussed more in Section 5(c). Moreover, our instability criteria for cases
when $n^{-}\left(  \mathcal{L}_{0}\right)  \geq2$ appear to be new, even for
the relatively well-studied BBM and KDV type equations. The situation
$n^{-}\left(  \mathcal{L}_{0}\right)  \geq2$ might arise for highly
oscillatory solitary waves (i.e. \cite{amt2}, \cite{abr-benjamin}). Even for
single-humped and positive solitary waves, it is not necessarily true that
$n^{-}\left(  \mathcal{L}_{0}\right)  =1$ since there is no Sturm theory for
general operators $\mathcal{M}$. One such example is the large solitary waves
for the full water wave problem. In \cite{lin-soli-water}, a similar
instability criterion is derived for solitary waves, in terms of an operator
$\mathcal{L}_{0}$ with $\alpha\left(  k\right)  =k\coth\left(  kH\right)  $,
for which $n^{-}\left(  \mathcal{L}_{0}\right)  $ grows without bound as the
solitary wave approaches the highest wave.

Let us discuss some implications of our results for solitary wave stability.
The solitary waves of regularized Boussinesq equations are known
(\cite{smereka}, \cite{pw92-evans}) to be highly indefinite (constrainted)
energy saddles, and therefore their stability can not be pursued by showing
energy minimizers as in the BBM and KDV cases. More interestingly, solitary
waves of the full water wave are also indefinite (constrainted) energy saddles
(\cite{bona-sachs}, \cite{hur-lin}) and thus the study of stability of RBou
solitary waves might shed some light on the full water wave problem. We note
that energy saddles are not necessarily unstable. Indeed, it is shown in
\cite{pw-97} that small solitary waves of the regularized Boussinesq equation
are spectrally stable, that is, there are no growing modes to the linearized
equation. So far, we do not know any method to prove nonlinear stability for
energy saddle type solutions. The spectral stability is naturally the first
step. The next theorem might be useful in the study of the spectral stability,
in particular, for large solitary waves of RBou type equations.

\begin{theorem}
\label{thm-transition}Consider solitary waves $u_{c}\left(  x-ct\right)  \ $of
equations (\ref{BBM})-(\ref{RBOU}), and assume $\ker\mathcal{L}_{0}=\left\{
u_{cx}\right\}  $. Suppose all possible growing modes are purely growing and
the spectral stability exchanges at $c_{0}$, then $P^{\prime}\left(
c_{0}\right)  =0$.
\end{theorem}

For the original regularized Boussinesq equation, it is shown in \cite[p.
79]{pw92-evans} that $P^{\prime}\left(  c\right)  >0$ for any $c^{2}>1$. By
Theorem \ref{thm-transition} and the spectral stability of small solitary
waves \cite{pw-97}, it follows that either all solitary waves are spectrally
stable or there is oscillatory instability for some solitary waves. So the
spectral stability of large solitary waves would follow if one could exclude
the oscillatory instability, namely, show that any growing mode must be purely
growing. For BBM and KDV type equations, when $n^{-}\left(  \mathcal{L}%
_{0}\right)  \geq2$, the solitary waves are also of energy saddle type and
their stability could not studied by the usual energy argument. Above remarks
also apply to these cases. We note that for KDV and BBM equations, under the
hypothesis (\ref{kernel-1}) the oscillatory instability can be excluded as in
the case $\mathcal{M}=-\partial_{x}^{2}$\ (\cite[p. 79]{pw92-evans}), by
adapting the finite-dimensional argument of \cite{mackay}.

We briefly discuss the proof of Theorem \ref{THM: main}. The growing modes
equations (\ref{spectral-BBM}), (\ref{spectral-Rbou}) and (\ref{spectrum-KDV})
are non-self-adjoint eigenvalue problems for variable coefficient operators
and rather few systematic techniques are available to study such problems. Our
key step is to reformulate the spectral problems in terms of a family of
operators $\mathcal{A}^{\lambda}$, which has the form of $\mathcal{M}$ plus
some nonlocal but bounded terms. The idea is to try to relate the eigenvalue
problems to the elliptic type problems for solitary waves. The existence of a
purely growing mode is equivalent to find some $\lambda>0$ such that
$\mathcal{A}^{\lambda}$ has a nontrivial kernel. This is achieved by a
continuation strategy to exploit the difference of the spectra of
$\mathcal{A}^{\lambda}$ near infinity and zero. First, we show that the
essential spectrum of $\mathcal{A}^{\lambda}$ lies to the right and away from
the imaginary axis. For large $\lambda$, the spectra of the operator
$\mathcal{A}^{\lambda}$ is shown to lie entirely in the right half complex
plane. So if for small $\lambda$, the operator $\mathcal{A}^{\lambda}$ has an
odd number of eigenvalues in the left half plane, then the spectrum of
$\mathcal{A}^{\lambda}$ must get across the origin at some $\lambda>0$ where a
purely growing mode appears. The zero-limit operator $\mathcal{A}^{0}$ is
exactly the operator $\mathcal{L}_{0}$. Since the convergence of
$\mathcal{A}^{\lambda}$ to $\mathcal{L}_{0}$ is rather weak, the usual
perturbation theory does not apply and the asymptotic perturbation theory by
Vock and Hunziker (\cite{vock-hunz82}) is used to study perturbations of the
eigenvalues of $\mathcal{L}_{0}$. In particular, it is important to know how
the zero eigenvalue of $\mathcal{L}_{0}$ is perturbed, for which we derive a
moving kernel formula. The instability criteria and Theorem
\ref{thm-transition} about the transition points follows from this formula.
One important technical issue in the proof is to use the decay of solitary
waves to obtain a priori estimates and gain certain compactness.

The approach of using nonlocal dispersion operators $\mathcal{A}^{\lambda}$
with continuation to find instability criteria originates from our previous
works (\cite{lin-sima}, \cite{lin-cmp}, \cite{lin-bgk}) on 2D ideal fluid and
1D electrostatic plasma, which have also been extended to study instability of
galaxies \cite{guo-lin} and 3D electromagnetic plasmas \cite{lin-strauss1},
\cite{lin-strauss2}. The consideration of the movement of $\ker\mathcal{A}%
^{0}\ $is suggested in \cite[Remark 3.2]{lin-cmp}. The techniques developed in
this paper have been extended to get stability criteria for periodic
dispersive waves (\cite{lin-periodic}), and to prove instability of large
solitary and periodic waves for the full water wave problem
(\cite{lin-soli-water}, \cite{lin-stokes}). This general approach might also
be useful for to study instability in dispersive wave systems and
multi-dimensional problems, which have been poorly understood.

This paper is organized as follows. In Section 2, we give details of the proof
of Theorem \ref{THM: main} for the BBM case. Section 3 treats the RBou case,
whose proof is rather similar to the BBM case. The KDV case has some subtle
difference to the previous cases and is discussed in Section 4. In Section 5,
we discuss some extensions and open issues. The Appendix gives an alternative
way of modifying the nonlinear instability proof in \cite{gss87} to general
dispersive long wave models.

\section{The BBM type equations}

Consider a traveling solution $u\left(  x,t\right)  =u_{c}\left(  x-ct\right)
\ \left(  c>1\right)  $ of the BBM type equation (\ref{BBM}). Then $u_{c}$
satisfies the equation%
\begin{equation}
\mathcal{M}u_{c}+\left(  1-\frac{1}{c}\right)  u_{c}-\frac{1}{c}f\left(
u_{c}\right)  =0\text{.} \label{solitary-bbm}%
\end{equation}
We define the following operator $\mathcal{L}_{0}:H^{m}\rightarrow L^{2}\ $by
the linearization of (\ref{solitary-bbm})%

\begin{equation}
\mathcal{L}_{0}=\mathcal{M}+\left(  1-\frac{1}{c}\right)  -\frac{1}%
{c}f^{\prime}\left(  u_{c}\right)  . \label{L0-BBM}%
\end{equation}
The linearized equation in the traveling frame $\left(  x-ct,t\right)  $ is
\begin{equation}
\left(  \partial_{t}-c\partial_{x}\right)  \left(  u+\mathcal{M}u\right)
+\partial_{x}\left(  u+f^{\prime}\left(  u_{c}\right)  u\right)  =0.
\label{L-bbm}%
\end{equation}
For a growing mode solution $e^{\lambda t}u\left(  x\right)  $ $\left(
\operatorname{Re}\lambda>0\right)  $ of (\ref{L-bbm}), $u\left(  x\right)  $
satisfies
\begin{equation}
\left(  \lambda-c\partial_{x}\right)  \left(  u+\mathcal{M}u\right)
+\partial_{x}\left(  u+f^{\prime}\left(  u_{c}\right)  u\right)  =0,
\label{spectral-BBM}%
\end{equation}
which can be written as
\[
\mathcal{M}u+u+\frac{\partial_{x}}{\lambda-c\partial_{x}}\left(  u+f^{\prime
}\left(  u_{c}\right)  u\right)  =0.
\]
This motivates us to define a family of operators $\mathcal{A}^{\lambda}%
:H^{m}\rightarrow L^{2}$ by
\[
\mathcal{A}^{\lambda}u=\mathcal{M}u+u+\frac{\partial_{x}}{\lambda
-c\partial_{x}}\left(  u+f^{\prime}\left(  u_{c}\right)  u\right)  .
\]
Thus the existence of a growing mode is reduced to find $\lambda\in\mathbb{C}$
with $\operatorname{Re}\lambda>0$ such that the operator $\mathcal{A}%
^{\lambda}$ has a nontrivial kernel. Below, we seek a purely growing mode with
$\lambda>0$. We use a continuation strategy, by exploiting the difference of
the spectra of the operators $\mathcal{A}^{\lambda}$ for $\lambda$ near
infinity and zero. We divide the proof into several steps.

\subsection{The properties of $\mathcal{A}^{\lambda}$}

Define the following operators
\[
\mathcal{D}=c\partial_{x},\ \mathcal{E}^{\lambda,\pm}=\frac{\mathcal{\lambda}%
}{\lambda\pm\mathcal{D}}.
\]
Then the operator $\mathcal{A}^{\lambda}\left(  \lambda>0\right)  $ can be
written as
\[
\mathcal{A}^{\lambda}=\mathcal{M}+1-\frac{1}{c}\left(  1-\mathcal{E}%
^{\lambda,-}\right)  \left(  1+f^{\prime}\left(  u_{c}\right)  \right)  .
\]

\begin{lemma}
\label{lemma-e-lb}

\textrm{{(a)} }For $\lambda>0,\ $\textrm{the operators }$\mathcal{E}%
^{\lambda,\pm}$\textrm{ are continuous in $\lambda$ and
\begin{equation}
\left\Vert \mathcal{E}^{\lambda,\pm}\right\Vert _{L^{2}\rightarrow L^{2}}%
\leq1,\ \label{estimate-E-lb}%
\end{equation}
}%
\begin{equation}
\left\Vert 1-\mathcal{E}^{\lambda,\pm}\right\Vert _{L^{2}\rightarrow L^{2}%
}\leq1. \label{estimate-E-lb-one}%
\end{equation}

\textrm{{(b)} When }$\lambda\rightarrow0+$, $\mathcal{E}^{\lambda,\pm}$
converges to $0$ strongly in $L^{2}\left(  \mathbf{R}\right)  $.

(c) When $\lambda\rightarrow+\infty,$ $\mathcal{E}^{\lambda,\pm}$ converges to
$1$ strongly in $L^{2}\left(  \mathbf{R}\right)  $.
\end{lemma}

\begin{proof}
We have
\[
\left\Vert \mathcal{E}^{\lambda,\pm}\phi\right\Vert _{L^{2}}^{2}%
=\int_{\mathbb{R}}\left\vert \frac{\lambda}{\lambda\pm ick}\right\vert
^{2}\left\vert \hat{\phi}\left(  k\right)  \right\vert ^{2}dk\leq
\int_{\mathbb{R}}\left\vert \hat{\phi}\left(  k\right)  \right\vert
^{2}dk=\left\Vert \phi\right\Vert _{L^{2}}^{2}%
\]
and (\ref{estimate-E-lb}) follows. Similarly, we get the estimate
(\ref{estimate-E-lb-one}). By the dominant convergence theorem,
\[
\left\Vert \mathcal{E}^{\lambda,\pm}\phi\right\Vert _{L^{2}}^{2}%
=\int_{\mathbb{R}}\left\vert \frac{\lambda}{\lambda\pm ick}\right\vert
^{2}\left\vert \hat{\phi}\left(  k\right)  \right\vert ^{2}dk\rightarrow0,
\]
when $\lambda\rightarrow0+$. Thus $\mathcal{E}^{\lambda,\pm}\rightarrow0$
strongly in $L^{2}$. The proof of (c) is similar to that of (b) and we skip it.
\end{proof}

\begin{corollary}
\label{Cor-A-lb-convergen}For $\lambda>0$, the operator $\mathcal{A}^{\lambda
}$ converges to $\mathcal{L}_{0}$ strongly in $L^{2}$ when $\lambda
\rightarrow0+$, and converges to $\mathcal{M}+1$ strongly in $L^{2}$ when
$\lambda\rightarrow+\infty$.
\end{corollary}

The following theorem states that the essential spectrum of $\mathcal{A}%
^{\lambda}$ is to the right and away from the imaginary axis.

\begin{proposition}
\label{prop-essential}For any $\lambda>0$, we have
\begin{equation}
\sigma_{\text{ess}}\left(  \mathcal{A}^{\lambda}\right)  \subset\left\{
z\ |\ \operatorname{Re}\lambda\geq\frac{1}{2}\left(  1-\frac{1}{c}\right)
>0\right\}  . \label{bound-essential}%
\end{equation}

\end{proposition}

The proof of Proposition \ref{prop-essential} is based on the following lemmas.

\begin{lemma}
\label{lemma-quadrtic-bound-lb}Consider any sequence
\[
\left\{  u_{n}\right\}  \in H^{m}\left(  \mathbf{R}\right)  ,\ \left\Vert
u_{n}\right\Vert _{2}=1,\ supp\ u_{n}\subset\left\{  x|\ \left\vert
x\right\vert \geq n\right\}  .
\]
Then for any complex number $z\ $with $\operatorname{Re}z<\frac{1}{2}\left(
1-\frac{1}{c}\right)  $, we have
\[
\operatorname{Re}\left(  \left(  \mathcal{A}^{\lambda}-z\right)  u_{n}%
,u_{n}\right)  \geq\frac{1}{4}\left(  1-\frac{1}{c}\right)  ,
\]
when $n$ is large enough.
\end{lemma}

\begin{proof}
We have
\begin{align*}
&  \ \ \ \ \operatorname{Re}\left(  \left(  \mathcal{A}^{\lambda}-z\right)
u_{n},u_{n}\right) \\
&  =\left(  \left(  \mathcal{M}+1\right)  u_{n},u_{n}\right)
-\operatorname{Re}z-\operatorname{Re}\left(  \frac{1}{c}\left(  1-\mathcal{E}%
^{\lambda,-}\right)  \left(  1+f^{\prime}\left(  u_{c}\right)  \right)
u_{n},u_{n}\right) \\
\ \ \  &  =\left(  \left(  \mathcal{M}+1\right)  u_{n},u_{n}\right)
-\operatorname{Re}z-\frac{1}{c}\operatorname{Re}\left(  \left(  1+f^{\prime
}\left(  u_{c}\right)  \right)  u_{n},\left(  1-\mathcal{E}^{\lambda
,+}\right)  u_{n}\right) \\
&  \geq1-\frac{1}{2}\left(  1-\frac{1}{c}\right)  -\frac{1}{c}\left(
1+\max_{\left\vert x\right\vert \geq n}\left\vert f^{\prime}\left(
u_{c}\right)  \right\vert \right)  \left\Vert \left(  1-\mathcal{E}%
^{\lambda,+}\right)  u_{n}\right\Vert _{2}\\
&  \geq\frac{1}{2}\left(  1-\frac{1}{c}\right)  -\frac{1}{c}\max_{\left\vert
x\right\vert \geq n}\left\vert f^{\prime}\left(  u_{c}\right)  \right\vert
\ \text{(by Lemma \ref{lemma-e-lb} (a))}\\
&  \geq\frac{1}{4}\left(  1-\frac{1}{c}\right)  ,\text{ when }n\text{ is big
enough.}%
\end{align*}

\end{proof}

To study the essential spectrum of $\mathcal{A}^{\lambda}$, first we introduce
the Zhislin Spectrum $Z\left(  \mathcal{A}^{\lambda}\right)  $
(\cite{hislop-sig-book}). A Zhislin sequence for $\mathcal{A}^{\lambda}$ and
$z\in\mathbb{C}$ is a sequence
\[
\left\{  u_{n}\right\}  \in H^{m},\ \left\Vert u_{n}\right\Vert _{2}%
=1,\ supp\ u_{n}\subset\left\{  x|\ \left\vert x\right\vert \geq n\right\}
\]
and $\left\Vert \left(  \mathcal{A}^{\lambda}-z\right)  u_{n}\right\Vert
_{2}\rightarrow0$ as $n\rightarrow\infty$. The set of all $z$ such that a
Zhislin sequence exists for $\mathcal{A}^{\lambda}$ and $z$ is denoted
$Z\left(  \mathcal{A}^{\lambda}\right)  $. From the above definition and Lemma
\ref{lemma-quadrtic-bound-lb}, we readily have
\begin{equation}
Z\left(  \mathcal{A}^{\lambda}\right)  \subset\left\{  z\in\mathbb{C}%
|\ \operatorname{Re}z\geq\frac{1}{2}\left(  1-\frac{1}{c}\right)  \right\}
\text{.} \label{bound-zhislin}%
\end{equation}
Another related spectrum is the Weyl spectrum $W\left(  \mathcal{A}^{\lambda
}\right)  $ (\cite{hislop-sig-book}). A Weyl sequence for $\mathcal{A}%
^{\lambda}$ and $z\in\mathbb{C\ }$\ is a sequence $\left\{  u_{n}\right\}  \in
H^{m},\left\Vert u_{n}\right\Vert _{2}=1,\ u_{n}\rightarrow0$ weakly in
$L^{2}$ and $\left\Vert \left(  \mathcal{A}^{\lambda}-z\right)  u_{n}%
\right\Vert _{2}\rightarrow0$ as $n\rightarrow\infty$. The set $W\left(
\mathcal{A}^{\lambda}\right)  $ is all $z\in\mathbb{C}$ such that a Weyl
sequence exists for $\mathcal{A}^{\lambda}$ and $z$. By (\cite[Theorem
10.10]{hislop-sig-book}), $W\left(  \mathcal{A}^{\lambda}\right)
\subset\sigma_{\text{ess}}\left(  \mathcal{A}^{\lambda}\right)  $ and the
boundary of $\sigma_{\text{ess}}\left(  \mathcal{A}^{\lambda}\right)  $ is
contained in $W\left(  \mathcal{A}^{\lambda}\right)  $. So it suffices to show
that $W\left(  \mathcal{A}^{\lambda}\right)  =Z\left(  \mathcal{A}^{\lambda
}\right)  $, which together with (\ref{bound-zhislin}) implies
(\ref{bound-essential}). By (\cite[Theorem 10.12]{hislop-sig-book}), the proof
of $W\left(  \mathcal{A}^{\lambda}\right)  =Z\left(  \mathcal{A}^{\lambda
}\right)  $ is reduced to prove the following lemma.

\begin{lemma}
\label{lemma-commu-d}Given $\lambda>0$. Let $\chi\in C_{0}^{\infty}\left(
\mathbf{R}\right)  $ be a cut-off function such that $\chi|_{\left\{
\left\vert x\right\vert \leq R_{0}\right\}  }=1$, for some $R_{0}>0$. Define
$\chi_{d}=\chi\left(  x/d\right)  ,\ d>0.$ Then for each $d,\ \chi_{d}\left(
\mathcal{A}^{\lambda}-z\right)  ^{-1}$ is compact for some $z\in\rho\left(
\mathcal{A}^{\lambda}\right)  $, and that there exists $C\left(  d\right)
\rightarrow0$ as $d\rightarrow\infty$ such that for any $u\in C_{0}^{\infty
}\left(  \mathbf{R}\right)  $,
\begin{equation}
\left\Vert \left[  \mathcal{A}^{\lambda},\chi_{d}\right]  u\right\Vert
_{2}\leq C\left(  d\right)  \left(  \left\Vert \mathcal{A}^{\lambda
}u\right\Vert _{2}+\left\Vert u\right\Vert _{2}\right)  .
\label{estimate-comm-d}%
\end{equation}

\end{lemma}

\begin{proof}
We write $\mathcal{A}^{\lambda}=\mathcal{M}+1+\mathcal{K}^{\lambda}$, where
\begin{equation}
\mathcal{K}^{\lambda}=\frac{1}{c}\left(  1-\mathcal{E}^{\lambda,-}\right)
\left(  1+f^{\prime}\left(  u_{c}\right)  \right)  :L^{2}\rightarrow L^{2}
\label{defn-cal-K}%
\end{equation}
is bounded. So $-k\in\rho\left(  \mathcal{A}^{\lambda}\right)  $ when $k>0$ is
sufficiently large. The compactness of $\chi_{d}\left(  \mathcal{A}^{\lambda
}+k\right)  ^{-1}$ is a corollary of the local compactness of $H^{m}%
\hookrightarrow L^{2}$. To show (\ref{estimate-comm-d}), we note that the
graph norm of $\mathcal{A}^{\lambda}$ is equivalent to $\left\Vert
\cdot\right\Vert _{H^{m}}$. Below, we use $C$ to denote a generic constant.
First, we have
\begin{align*}
\left[  \mathcal{K}^{\lambda},\chi_{d}\right]   &  =-\frac{1}{c}\left[
\mathcal{E}^{\lambda,-},\chi_{d}\right]  \left(  1+f^{\prime}\left(
u_{c}\right)  \right)  =-\frac{1}{c}\frac{\mathcal{\lambda}}{\lambda
-\mathcal{D}}\left[  \mathcal{D},\chi_{d}\right]  \frac{1}{\lambda
-\mathcal{D}}\left(  1+f^{\prime}\left(  u_{c}\right)  \right) \\
&  =\frac{1}{\lambda cd}\mathcal{E}^{\lambda,-}\chi^{\prime}\left(
x/d\right)  \mathcal{E}^{\lambda,-}\left(  1+f^{\prime}\left(  u_{c}\right)
\right)
\end{align*}
and thus
\begin{equation}
\left\Vert \left[  \mathcal{K}^{\lambda},\chi_{d}\right]  \right\Vert
_{L^{2}\rightarrow L^{2}}\leq\frac{C}{\lambda d}. \label{inter1}%
\end{equation}
Let $l=\left[  m\right]  $ to be the largest integer no greater than $m$ and
$\delta=m-\left[  m\right]  \in\lbrack0,1)$. Define the following two
operators
\begin{equation}
\mathcal{M}_{1}=\left\{
\begin{array}
[c]{cc}%
1+\left(  \frac{d}{dx}\right)  ^{l} & \text{if }l\neq2\ \operatorname{mod}%
\ 4\\
1-\left(  \frac{d}{dx}\right)  ^{l} & \text{if }l=2\ \operatorname{mod}\ 4.
\end{array}
\right.  \label{defn-m1}%
\end{equation}
and $\mathcal{M}_{2}$ is the Fourier multiplier operator with the symbol
\begin{equation}
n\left(  k\right)  =\left\{
\begin{array}
[c]{cc}%
\frac{\alpha\left(  k\right)  }{1+\left(  ik\right)  ^{l}} & \text{if }%
l\neq2\operatorname{mod}4\\
\frac{\alpha\left(  k\right)  }{1-\left(  ik\right)  ^{l}} & \text{if
}l=2\operatorname{mod}4.
\end{array}
\right.  . \label{defn-m2}%
\end{equation}
Then $\mathcal{M}=\mathcal{M}_{2}\mathcal{M}_{1}$ and
\[
\left[  \mathcal{M},\chi_{d}\right]  =\mathcal{M}_{2}\left[  \mathcal{M}%
_{1},\chi_{d}\right]  +\left[  \mathcal{M}_{2},\chi_{d}\right]  \mathcal{M}%
_{1}.
\]
We study $\left[  \mathcal{M}_{2},\chi_{d}\right]  $ in two cases. When
$\delta=0$, that is, $m$ is an integer, for any $v\in C_{0}^{\infty}\left(
\mathbf{R}\right)  $, we follow \cite[P.127-128]{cordes} to write
\begin{align*}
\left[  \mathcal{M}_{2},\chi_{d}\right]  v  &  =-\left(  2\pi\right)
^{-\frac{1}{2}}\int\check{n}\left(  x-y\right)  \left(  \chi_{d}\left(
x\right)  -\chi_{d}\left(  y\right)  \right)  v\left(  y\right)  dy\\
&  =-\int_{0}^{1}\int\left(  2\pi\right)  ^{-\frac{1}{2}}\left(  x-y\right)
\check{n}\left(  x-y\right)  \chi_{d}^{\prime}\left(  \rho\left(  x-y\right)
+y\right)  v\left(  y\right)  dyd\rho\\
&  =\int_{0}^{1}A_{\rho}v\ d\rho,
\end{align*}
where $A_{\rho}$ is the integral operator with the kernel function
\[
K_{\rho}\left(  x,y\right)  =-\left(  2\pi\right)  ^{-\frac{1}{2}}\left(
x-y\right)  \check{n}\left(  x-y\right)  \chi_{d}^{\prime}\left(  \rho\left(
x-y\right)  +y\right)  .
\]
Note that $\beta\left(  x\right)  =x\check{n}\left(  x\right)  $ is the
inverse Fourier transformation of $in^{\prime}\left(  k\right)  $ and
$n^{\prime}\left(  k\right)  \in L^{2}\ $when $l=m$, so $\beta\left(
x\right)  \in L^{2}$. Thus
\begin{align*}
\int\int\left\vert K_{\rho}\left(  x,y\right)  \right\vert ^{2}dxdy  &
=2\pi\int\int\left\vert \beta\right\vert ^{2}\left(  x-y\right)  \left\vert
\chi_{d}^{\prime}\right\vert ^{2}\left(  \rho\left(  x-y\right)  +y\right)
\ dxdy\\
&  =2\pi\int\int\left\vert \beta\right\vert ^{2}\left(  x\right)  \left\vert
\chi_{d}^{\prime}\right\vert ^{2}\left(  y\right)  \ dxdy=2\pi\left\Vert
\beta\right\Vert _{L_{2}}^{2}\left\Vert \chi_{d}^{\prime}\right\Vert _{L^{2}%
}^{2}\\
&  =\frac{2\pi}{d}\left\Vert \beta\right\Vert _{L_{2}}^{2}\left\Vert
\chi^{\prime}\right\Vert _{L^{2}}^{2}.
\end{align*}
So
\[
\left\Vert \left[  \mathcal{M}_{2},\chi_{d}\right]  \right\Vert _{L^{2}%
\rightarrow L^{2}}\leq\frac{C}{d^{\frac{1}{2}}}%
\]
and
\[
\left\Vert \left[  \mathcal{M}_{2},\chi_{d}\right]  \mathcal{M}_{1}%
u\right\Vert _{L^{2}}\leq\frac{C}{d^{\frac{1}{2}}}\left\Vert \mathcal{M}%
_{1}u\right\Vert _{L^{2}}\leq\frac{C}{d^{\frac{1}{2}}}\left\Vert u\right\Vert
_{H^{m}}.
\]
When $\delta>0$, we define two Fourier multiplier operators $\mathcal{M}_{3}$
and $\mathcal{M}_{4}$ with symbols $1+\left\vert k\right\vert ^{\delta}$ and
$n_{1}\left(  k\right)  =n\left(  k\right)  /\left(  1+\left\vert k\right\vert
^{\delta}\right)  $ respectively. Then $\mathcal{M}_{2}=\mathcal{M}%
_{3}\mathcal{M}_{4}$ and
\[
\left[  \mathcal{M}_{2},\chi_{d}\right]  =\mathcal{M}_{4}\left[
\mathcal{M}_{3},\chi_{d}\right]  +\left[  \mathcal{M}_{4},\chi_{d}\right]
\mathcal{M}_{3}.
\]
Since $n_{1}^{\prime}\left(  k\right)  \in L^{2}$, by the same argument as
above, we have
\[
\left\Vert \left[  \mathcal{M}_{4},\chi_{d}\right]  \right\Vert _{L^{2}%
\rightarrow L^{2}}\leq\frac{C}{d^{\frac{1}{2}}}.
\]
By \cite[P. 213, Theorem 3.3]{murray},
\[
\left\Vert \left[  \mathcal{M}_{3},\chi_{d}\right]  \right\Vert _{L^{2}%
\rightarrow L^{2}}\leq C\left(  \delta\right)  \left\Vert \left\vert
D\right\vert ^{\delta}\chi_{d}\right\Vert _{\ast},
\]
where $\left\vert D\right\vert ^{\delta}$ is the fractional differentiation
operator with the symbol $\left\vert k\right\vert ^{\delta}$ and $\left\Vert
\cdot\right\Vert _{\ast}$ is the BMO norm. By using Fourier transformations,
it is easy to check that
\[
\left(  \left\vert D\right\vert ^{\delta}\chi_{d}\right)  \left(  x\right)
=\frac{1}{d^{\delta}}\left(  \left\vert D\right\vert ^{\delta}\chi\right)
\left(  \frac{x}{d}\right)  .
\]
So%
\[
\left\Vert \left\vert D\right\vert ^{\delta}\chi_{d}\right\Vert _{\ast}%
\leq2\left\Vert \left\vert D\right\vert ^{\delta}\chi_{d}\right\Vert
_{L^{\infty}}\leq\frac{2}{d^{\delta}}\left\Vert \left\vert D\right\vert
^{\delta}\chi\right\Vert _{L^{\infty}}%
\]
and therefore
\[
\left\Vert \left[  \mathcal{M}_{3},\chi_{d}\right]  \right\Vert _{L^{2}%
\rightarrow L^{2}}\leq\frac{C}{d^{\delta}}.
\]
Since $\left\Vert \mathcal{M}_{4}\right\Vert _{L^{2}\rightarrow L^{2}}$ is
bounded, we have
\begin{align*}
\left\Vert \left[  \mathcal{M}_{2},\chi_{d}\right]  \mathcal{M}_{1}%
u\right\Vert _{L^{2}}  &  \leq\left\Vert \mathcal{M}_{4}\left[  \mathcal{M}%
_{3},\chi_{d}\right]  \mathcal{M}_{1}u\right\Vert _{L^{2}}+\left\Vert \left[
\mathcal{M}_{4},\chi_{d}\right]  \mathcal{M}_{3}\mathcal{M}_{1}u\right\Vert
_{L^{2}}\\
&  \leq\frac{C}{d^{\delta}}\left\Vert \mathcal{M}_{1}u\right\Vert _{L^{2}%
}+\frac{C}{d^{\frac{1}{2}}}\left\Vert \mathcal{M}_{3}\mathcal{M}%
_{1}u\right\Vert _{L^{2}}\leq C\left(  \frac{1}{d^{\delta}}+\frac{1}%
{d^{\frac{1}{2}}}\right)  \left\Vert u\right\Vert _{H^{m}}.
\end{align*}
So in both cases,
\begin{equation}
\left\Vert \left[  \mathcal{M}_{2},\chi_{d}\right]  \mathcal{M}_{1}%
u\right\Vert _{L^{2}}\leq C\left(  d\right)  \left\Vert u\right\Vert _{H^{m}%
}\text{, with }C\left(  d\right)  \rightarrow0\text{ as }d\rightarrow
\infty\text{. } \label{inter3}%
\end{equation}
Since
\[
\left[  \mathcal{M}_{1},\chi_{d}\right]  =\sum_{j=1}^{l}C_{j}^{l}\frac
{d^{j}\chi_{d}}{dx^{j}}\frac{d^{l-j}}{dx^{l-j}}\text{ or }-\sum_{j=1}^{l}%
C_{j}^{l}\frac{d^{j}\chi_{d}}{dx^{j}}\frac{d^{l-j}}{dx^{l-j}},
\]
and
\[
\frac{d^{j}\chi_{d}}{dx^{j}}\left(  x\right)  =\frac{1}{d^{j}}\chi^{\left(
j\right)  }\left(  \frac{x}{d}\right)  :=\frac{1}{d^{j}}\chi_{d}^{\left(
j\right)  },
\]
we have
\[
\left\Vert \mathcal{M}_{2}\left[  \mathcal{M}_{1},\chi_{d}\right]
u\right\Vert _{2}\leq\sum_{j=1}^{l}\frac{C_{j}^{l}}{d^{j}}\left(  \left\Vert
\left[  \mathcal{M}_{2},\chi_{d}^{\left(  j\right)  }\right]  u^{\left(
l-j\right)  }\right\Vert _{L^{2}}+\left\Vert \chi\right\Vert _{C^{l}%
}\left\Vert \mathcal{M}_{2}u^{\left(  l-j\right)  }\right\Vert _{L^{2}%
}\right)  .
\]
By similar estimates as above, when $\delta=0$,
\[
\left\Vert \left[  \mathcal{M}_{2},\chi_{d}^{\left(  j\right)  }\right]
u^{\left(  l-j\right)  }\right\Vert _{L^{2}}\leq\frac{C}{d^{\frac{1}{2}}%
}\left\Vert u^{\left(  l-j\right)  }\right\Vert _{L^{2}}\leq\frac{C}%
{d^{\frac{1}{2}}}\left\Vert u\right\Vert _{H^{m}}%
\]
and when $\delta>0$,
\[
\left\Vert \left[  \mathcal{M}_{2},\chi_{d}^{\left(  j\right)  }\right]
u^{\left(  l-j\right)  }\right\Vert _{L^{2}}\leq\frac{C}{d^{\delta}}\left\Vert
u^{\left(  l-j\right)  }\right\Vert _{L^{2}}+\frac{C}{d^{\frac{1}{2}}%
}\left\Vert \mathcal{M}_{3}u^{\left(  l-j\right)  }\right\Vert _{L^{2}}\leq
C\left(  \frac{1}{d^{\delta}}+\frac{1}{d^{\frac{1}{2}}}\right)  \left\Vert
u\right\Vert _{H^{m}}.
\]
Thus
\[
\left\Vert \mathcal{M}_{2}\left[  \mathcal{M}_{1},\chi_{d}\right]
u\right\Vert _{2}\leq C\left(  d\right)  \left\Vert u\right\Vert _{H^{m}%
}\text{, with }C\left(  d\right)  \rightarrow0\text{ as }d\rightarrow
\infty\text{. }%
\]
Combining above with (\ref{inter1}) and (\ref{inter3}), we get the estimate
(\ref{estimate-comm-d}). This finishes the proof of the lemma and Proposition
\ref{prop-essential}.
\end{proof}

To show the existence of growing modes, we need to find some $\lambda>0$ such
that $\mathcal{A}^{\lambda}$ has a nontrivial kernel. We use a continuation
strategy, by comparing the behavior of $\mathcal{A}^{\lambda}$ near $0$ and
infinity. First, we study the case near infinity. \ 

\begin{lemma}
\label{lemma-no-eigen-infy}There exists $\Lambda>0$, such that when
$\lambda>\Lambda$, $\mathcal{A}^{\lambda}$ has no eigenvalues in $\left\{
z|\ \operatorname{Re}z\leq0\right\}  $.
\end{lemma}

\begin{proof}
Suppose otherwise, then there exists a sequence $\left\{  \lambda_{n}\right\}
\rightarrow\infty$, and $\left\{  k_{n}\right\}  \in\mathbb{C},\left\{
u_{n}\right\}  \in$ $H^{m}\left(  \mathbf{R}\right)  $, such that
$\operatorname{Re}k_{n}\leq0$ and $\left(  \mathcal{A}^{\lambda_{n}}%
-k_{n}\right)  u_{n}=0$. Since $\left\Vert \mathcal{A}^{\lambda}%
-\mathcal{M-}1\right\Vert =\mathcal{\ }\left\Vert \mathcal{K}^{\lambda
}\right\Vert \leq M$ \ for some constant $M$ independent of $\lambda$ and
$\mathcal{M}$ is a self-adjoint positive operator, all discrete eigenvalues of
$\mathcal{A}^{\lambda}$ lie in
\[
D_{M}=\left\{  z|\ \operatorname{Re}z\geq-M\text{ and }\left\vert
\operatorname{Im}z\right\vert \leq M\right\}  .
\]
Therefore, $k_{n}\rightarrow$ $k_{\infty}\in D_{M}$ with $\operatorname{Re}%
k_{\infty}\leq0$. Denote $e\left(  x\right)  =\left(  f^{\prime}\left(
u_{c}\right)  \right)  ^{2}$, then $e\left(  x\right)  \rightarrow0$ when
$\left\vert x\right\vert \rightarrow\infty$. We normalize $u_{n}$ by setting
$\left\Vert u_{n}\right\Vert _{L_{e}^{2}}=1$, where
\begin{equation}
\left\Vert u\right\Vert _{L_{e}^{2}}=\left(  \int e\left(  x\right)
\left\vert u\right\vert ^{2}\ dx\right)  ^{\frac{1}{2}}. \label{defn-L2-e}%
\end{equation}
We claim that
\begin{equation}
\left\Vert u_{n}\right\Vert _{H^{\frac{m}{2}}}\leq C\text{, for a constant
}C\text{ independent of }n\text{.} \label{claim-bound-infy}%
\end{equation}
Assuming (\ref{claim-bound-infy}), we have $u_{n}\rightarrow u_{\infty}$
weakly in $H^{m}$. Moreover, $u_{\infty}\neq0$. To show that, we choose $R>0$
large enough such that $\max_{\left\vert x\right\vert \geq R}e\left(
x\right)  \leq\frac{1}{2C}$. Then
\[
\int_{\left\vert x\right\vert \geq R}e\left(  x\right)  \left\vert
u_{n}\right\vert ^{2}\ dx\leq\frac{1}{2C}\left\Vert u_{n}\right\Vert _{L^{2}%
}\leq\frac{1}{2}.
\]
Since $u_{n}\rightarrow u_{\infty}$ strongly in $L^{2}(\left\{  \left\vert
x\right\vert \leq R\right\}  )$, we have
\[
\int_{\left\vert x\right\vert \leq R}e\left(  x\right)  \left\vert u_{\infty
}\right\vert ^{2}dx=\lim_{n\rightarrow\infty}\int_{\left\vert x\right\vert
\leq R}e\left(  x\right)  \left\vert u_{n}\right\vert ^{2}dx\geq\frac{1}{2}%
\]
and thus $u_{\infty}\neq0$. By Corollary \ref{Cor-A-lb-convergen},
$\mathcal{A}^{\lambda_{n}}\rightarrow\mathcal{M}+1$ strongly in $L^{2}$,
therefore $\mathcal{A}^{\lambda_{n}}u_{n}\rightarrow$ $\left(  \mathcal{M}%
+1\right)  u_{\infty}$ weakly and $\left(  \mathcal{M}+1\right)  u_{\infty
}=k_{\infty}u_{\infty}$. Since $\operatorname{Re}k_{\infty}\leq0$, this a
contradiction. It remains to show (\ref{claim-bound-infy}). From $\left(
\mathcal{A}^{\lambda_{n}}-k_{n}\right)  u_{n}=0$, we get%
\begin{align*}
0  &  \geq\operatorname{Re}k_{n}\left\Vert u_{n}\right\Vert _{^{2}}%
^{2}=\left(  \left(  \mathcal{M}+1\right)  u_{n},u_{n}\right)  -\frac{1}%
{c}\operatorname{Re}\left(  u_{n},\left(  1-\mathcal{E}^{\lambda,+}\right)
u_{n}\right) \\
&  \ \ \ \ \ \ \ \ \ \ \ \ \ \ \ \ \ \ \ \ \ \ -\frac{1}{c}\operatorname{Re}%
\left(  f^{\prime}\left(  u_{c}\right)  u_{n},\left(  1-\mathcal{E}%
^{\lambda,+}\right)  u_{n}\right)  .
\end{align*}
By our assumption on the symbol $\alpha\left(  k\right)  \ $of $\mathcal{M}$,
there exists $K>0$ such that $\alpha\left(  k\right)  \geq a\left\vert
k\right\vert ^{m}$ when $\left\vert k\right\vert \geq K$. So for any
$\varepsilon,\delta>0$, from above and Lemma \ref{lemma-e-lb}, we have
\begin{align*}
0  &  \geq\left(  1-\delta\right)  \left\Vert u_{n}\right\Vert _{L^{2}}%
^{2}+a\int_{\left\vert k\right\vert \geq K}\left\vert k\right\vert
^{m}\left\vert \hat{u}\left(  k\right)  \right\vert ^{2}dk+\delta
\int_{\left\vert k\right\vert \leq K}\left\vert \hat{u}\left(  k\right)
\right\vert ^{2}dk\\
&  \ \ \ \ \ \ -\frac{1}{c}\left\Vert u_{n}\right\Vert _{L^{2}}^{2}-\frac
{1}{c}\left\Vert u_{n}\right\Vert _{L^{2}}\left\Vert u_{n}\right\Vert
_{L_{e}^{2}}\\
&  \geq\left(  1-\delta\right)  \left\Vert u_{n}\right\Vert _{L^{2}}^{2}%
+\min\left\{  \frac{\delta}{K^{m}},a\right\}  \int\left\vert k\right\vert
^{m}\left\vert \hat{u}\left(  k\right)  \right\vert ^{2}dk-\frac{1}%
{c}\left\Vert u_{n}\right\Vert _{L^{2}}^{2}\\
&  \ \ \ \ \ \ -\varepsilon\left\Vert u_{n}\right\Vert _{L^{2}}^{2}%
-\frac{\varepsilon}{4c^{2}}\left\Vert u_{n}\right\Vert _{L_{e}^{2}}^{2}\\
&  \geq\min\left\{  1-\frac{1}{c}-\delta-\varepsilon,\frac{\delta}{K^{m}%
},a\right\}  \left\Vert u_{n}\right\Vert _{H^{\frac{m}{2}}}^{2}-\frac
{\varepsilon}{4c^{2}}\left\Vert u_{n}\right\Vert _{L_{e}^{2}}^{2}%
\end{align*}
The bound (\ref{claim-bound-infy}) follows by choosing $\delta,\varepsilon>0$ small.
\end{proof}

\subsection{ Asymptotic perturbations near $\lambda=0$}

In this subsection, we study the spectra of $\mathcal{A}^{\lambda}$ for small
$\lambda$. When $\lambda\rightarrow0+,\ \mathcal{A}^{\lambda}\rightarrow
\mathcal{L}_{0}$ strongly in $L^{2}$, where $\mathcal{L}_{0}$ is defined by
(\ref{L0-BBM}). Since the convergence of $\mathcal{A}^{\lambda}\rightarrow
\mathcal{L}_{0}$ is rather weak, we could not use the regular perturbation
theory. Instead, we use the asymptotic perturbation theory developed by Vock
and Hunziker (\cite{vock-hunz82}), see also \cite{hislop-sig-book},
\cite{hunz-notes}. To apply this theory, we need some preliminary lemmas.

\begin{lemma}
\label{lemma-commu-F}Given $F\in C_{0}^{\infty}\left(  \mathbf{R}\right)  $.
Consider any sequence $\lambda_{n}\rightarrow0+$ and $\left\{  u_{n}\right\}
\in H^{m}\left(  \mathbf{R}\right)  $ satisfying
\begin{equation}
\left\Vert \mathcal{A}^{\lambda_{n}}u_{n}\right\Vert _{2}+\left\Vert
u_{n}\right\Vert _{2}\leq M_{1}<\infty\label{bound-graph-0}%
\end{equation}
for some constant $M_{1}$. Then if $w-\lim_{n\rightarrow\infty}u_{n}=0$, we
have
\begin{equation}
\lim_{n\rightarrow\infty}\left\Vert Fu_{n}\right\Vert _{2}=0 \label{lim1}%
\end{equation}
and
\begin{equation}
\lim_{n\rightarrow\infty}\left\Vert \left[  \mathcal{A}^{\lambda_{n}%
},F\right]  u_{n}\right\Vert _{2}=0. \label{lim2}%
\end{equation}

\end{lemma}

\begin{proof}
Since (\ref{bound-graph-0}) implies that $\left\Vert u_{n}\right\Vert _{H^{m}%
}\leq C$, (\ref{lim1}) follows from the local compactness of $H^{m}%
\hookrightarrow L^{2}$. To prove (\ref{lim2}), we use the notations in the
proof of Lemma \ref{lemma-commu-d}. We write $\mathcal{A}^{\lambda_{n}%
}=\mathcal{M+}1\mathcal{+K}^{\lambda_{n}}$. Note that
\[
\left[  \mathcal{M},F\right]  =\left[  \mathcal{M}_{2}\mathcal{M}%
_{1},F\right]  =\mathcal{M}_{2}\left[  \mathcal{M}_{1},F\right]  +\left[
\mathcal{M}_{2},F\right]  \mathcal{M}_{1},
\]
where $\mathcal{M}_{1}\ $and $\mathcal{M}_{2}$ are defined in (\ref{defn-m1})
and (\ref{defn-m2}). Let $G\in C_{0}^{\infty}\left(  \mathbf{R}\right)  $
satisfying $G=1$ on the support of $F$. For any $\varepsilon>0$, we have
\begin{align*}
\left\Vert \left[  \mathcal{M}_{1},F\right]  u_{n}\right\Vert _{2\text{ }}  &
=\left\Vert \left[  \mathcal{M}_{1},F\right]  Gu_{n}\right\Vert _{2\text{ }%
}=\left\Vert \sum_{j=1}^{l}C_{j}^{l}\frac{d^{j}F}{dx^{j}}\frac{d^{l-j}\left(
Gu_{n}\right)  }{dx^{l-j}}\text{ }\right\Vert _{2}\\
&  \leq C\left\Vert u_{n}\right\Vert _{H^{l-1}}\leq\varepsilon\left\Vert
u_{n}\right\Vert _{H^{m}}+C_{\varepsilon}\left\Vert Gu_{n}\right\Vert _{2}.
\end{align*}
Since $\varepsilon$ is arbitrarily small and the second term tends to zero by
the local compactness, it follows that $\left\Vert \left[  \mathcal{M}%
_{1},F\right]  u_{n}\right\Vert _{2}\rightarrow0$ when $n\rightarrow\infty$.
Since $n^{\prime}\left(  k\right)  \rightarrow0$ when $\left\vert k\right\vert
\rightarrow\infty$, by \cite[Theorem C]{cordes} the commutator $\left[
\mathcal{M}_{2},F\right]  :L^{2}\rightarrow L^{2}$ is compact. Since
$\left\Vert \mathcal{M}_{1}u_{n}\right\Vert _{2}\leq\left\Vert u_{n}%
\right\Vert _{H^{m}}\leq C$ and $u_{n}\rightarrow0$ weakly in $L^{2}$, we have
$\mathcal{M}_{1}u_{n}\rightarrow0$ weakly in $L^{2}$. So
\[
\left\Vert \left[  \mathcal{M}_{2},F\right]  \mathcal{M}_{1}u_{n}\right\Vert
_{2}\rightarrow0
\]
strongly in $L^{2}$ and thus $\left\Vert \left[  \mathcal{M},F\right]
u_{n}\right\Vert _{2\text{ }}\rightarrow0$. Since%
\begin{align*}
\left[  \mathcal{K}^{\lambda_{n}},F\right]  u_{n}  &  =-\frac{1}{c}\left[
\mathcal{E}^{\lambda_{n},-},F\right]  \left(  1+f^{\prime}\left(
u_{c}\right)  \right)  u_{n}\\
&  =-\frac{1}{c}\mathcal{E}^{\lambda_{n},-}F\left(  1+f^{\prime}\left(
u_{c}\right)  \right)  u_{n}+\frac{1}{c}F\mathcal{E}^{\lambda_{n},-}\left(
1+f^{\prime}\left(  u_{c}\right)  \right)  u_{n}=p_{n}+q_{n}.
\end{align*}
Denote $v_{n}=\left(  1+f^{\prime}\left(  u_{c}\right)  \right)  u_{n}$. From
the uniform bound of $\left\Vert u_{n}\right\Vert _{H^{m}}$, we get the
uniform bound for $\left\Vert v_{n}\right\Vert _{H^{m}}$. Therefore, by local
compactness,
\[
\left\Vert p_{n}\right\Vert _{2}\leq C\left\Vert Fv_{n}\right\Vert
_{2}\rightarrow0\text{, when }n\rightarrow\infty\text{. }%
\]
Since the operator $\left\Vert \mathcal{E}^{\lambda,-}\right\Vert
_{L^{2}\rightarrow L^{2}}\leq1$ and $\mathcal{E}^{\lambda,-}$ is commutable
with $\left(  1-\frac{d^{2}}{dx^{2}}\right)  ^{\frac{m}{2}}$, for any
$\lambda>0,\ $we have the estimate
\[
\left\Vert \mathcal{E}^{\lambda,-}\right\Vert _{H^{m}\rightarrow H^{m}}%
\leq1\text{. }%
\]
So denoting $\tilde{v}_{n}=\mathcal{E}^{\lambda_{n},-}v_{n},$we have the
uniform bound for$\ \left\Vert \tilde{v}_{n}\right\Vert _{H^{m}}$ and thus
\[
\left\Vert p_{n}\right\Vert _{2}\leq C\left\Vert F\tilde{v}_{n}\right\Vert
_{2}\rightarrow0\text{, when }n\rightarrow\infty\text{. }%
\]
This finishes the proof of (\ref{lim2}).
\end{proof}

\begin{lemma}
\label{lemma-quadratic bd-asymp}Let $z\in\mathbb{C}$ with $\operatorname{Re}%
z\leq\frac{1}{2}\left(  1-\frac{1}{c}\right)  $, then for some $n>0$ and all
$u\in C_{0}^{\infty}\left(  \left\vert x\right\vert \geq n\right)  $, we have
\begin{equation}
\left\Vert \left(  \mathcal{A}^{\lambda}-z\right)  u\right\Vert _{2}\geq
\frac{1}{4}\left(  1-\frac{1}{c}\right)  \left\Vert u\right\Vert _{2},
\label{resolvent-bound}%
\end{equation}
when $\lambda$ is sufficiently small.
\end{lemma}

\begin{proof}
The estimate (\ref{resolvent-bound}) follows from
\begin{equation}
\operatorname{Re}\left(  \left(  \mathcal{A}^{\lambda}-z\right)  u,u\right)
\geq\frac{1}{4}\left(  1-\frac{1}{c}\right)  \left\Vert u\right\Vert _{2}^{2},
\label{quadratic-bound-zero}%
\end{equation}
which can be obtained as in the proof of Lemma \ref{lemma-quadrtic-bound-lb}.
\end{proof}

With above two lemmas, we can apply the asymptotic perturbation theory
(\cite{hislop-sig-book}, \cite{vock-hunz82}) to get the eigenvalue
perturbations of $\mathcal{A}^{0}$ to $\mathcal{A}^{\lambda}\mathcal{\ }$with
small $\lambda$.

\begin{proposition}
\label{Prop-asym-perturb}Each discrete eigenvalue $k_{0}$ of $\mathcal{A}^{0}$
with $k_{0}\leq\frac{1}{2}\left(  1-\frac{1}{c}\right)  $ is stable with
respect to the family $\mathcal{A}^{\lambda}$ in the following sense: there
exists $\lambda_{1},\delta>0$, such that for $0<\lambda<\lambda_{1}$, we have

(i)
\[
B\left(  k_{0};\delta\right)  =\left\{  z|\ 0<\left\vert z-k_{0}\right\vert
<\delta\right\}  \subset P\left(  \mathcal{A}^{\lambda}\right)  ,
\]
where
\[
P\left(  \mathcal{A}^{\lambda}\right)  =\left\{  z|\ R^{\lambda}\left(
z\right)  =\left(  \mathcal{A}^{\lambda}-z\right)  ^{-1}\text{ exists and is
uniformly bounded for }\lambda\in\left(  0,\lambda_{1}\right)  \right\}  .
\]

(ii) Denote
\[
P_{\lambda}=\oint_{\left\{  \left\vert z-k_{0}\right\vert =\delta\right\}
}R^{\lambda}\left(  z\right)  \ dz\text{ and \ }P_{0}=\oint_{\left\{
\left\vert z-k_{0}\right\vert =\delta\right\}  }R^{0}\left(  z\right)  \ dz
\]
to be the perturbed and unperturbed spectral projection. Then $\dim
P_{\lambda}=\dim P_{0}$ and $\lim_{\lambda\rightarrow0}\left\Vert P_{\lambda
}-P_{0}\right\Vert =0.$
\end{proposition}

It follows from above that for $\lambda$ small, the operators $\mathcal{A}%
^{\lambda}$ have discrete eigenvalues inside $B\left(  k_{0};\delta\right)  $
with the total algebraic multiplicity equal to that of $k_{0}$.

\subsection{The Moving kernel formula and proof of instability\newline}

To understand the entire spectrum of $\mathcal{A}^{\lambda}$ for small
$\lambda$, we need to know precisely how the zero eigenvalue of $\mathcal{A}%
^{0}=\mathcal{L}^{0}$ is perturbed. For that, we derive a moving kernel
formula, from which the instability criterion follows. Let $\lambda_{1}%
,\delta>0$ be as in Proposition \ref{Prop-asym-perturb} for $k_{0}=0$. By our
assumption that $\ker\mathcal{A}^{0}=$ $\left\{  u_{cx}\right\}  $, so $\dim
P_{0}=1$ and thus $\dim P_{\lambda}=1$ for $\lambda<\lambda_{1}$. Since the
eigenvalues of $\mathcal{A}^{\lambda}$ appear in conjugate pairs, there is
only one real eigenvalue $k_{\lambda}\ $of $\mathcal{A}^{\lambda}$ inside
$B\left(  0;\delta\right)  $. The following lemma determines the sign of
$k_{\lambda},$ when $\lambda$ is sufficiently small.

\begin{lemma}
\label{lemma-moving}Assume $\ker\mathcal{L}^{0}=$ $\left\{  u_{cx}\right\}  $.
For $\lambda>0$ small enough, let $k_{\lambda}\in\mathbf{R}$ to be the only
eigenvalue of $\mathcal{A}^{\lambda}$ near origin. Then$\ $
\begin{equation}
\lim_{\lambda\rightarrow0+}\frac{k_{\lambda}}{\lambda^{2}}=-\frac{1}{c}%
\frac{dP}{dc}/\left\Vert u_{cx}\right\Vert _{L^{2}}^{2},
\label{formula-moving}%
\end{equation}
where the momentum
\begin{equation}
P\left(  c\right)  =\frac{1}{2}\left(  \left(  \mathcal{M+}1\right)
u_{c},u_{c}\right)  \label{P-BBM}%
\end{equation}

\end{lemma}

By the same proof of (\ref{claim-bound-infy}), we get the following a priori
estimate which is used in the later proof.

\begin{lemma}
\label{lemma-apriori}For $\lambda>0$ small enough, consider $u\in H^{m}\left(
\mathbf{R}\right)  $ satisfying the equation $\left(  \mathcal{A}^{\lambda
}-z\right)  u=v$, where $z\in\mathbb{C}$ with $\operatorname{Re}z\leq\frac
{1}{2}\left(  1-\frac{1}{c}\right)  $ and $v\in L^{2}$. Then we have the
estimate
\begin{equation}
\left\Vert u\right\Vert _{H^{\frac{m}{2}}}\leq C\left(  \left\Vert
u\right\Vert _{L_{e}^{2}}+\left\Vert v\right\Vert _{L^{2}}\right)  ,
\label{apriori-small-lb}%
\end{equation}
for some constant $C$ independent of $\lambda$. Here, the weighted norm
$\left\Vert \cdot\right\Vert _{L_{e}^{2}}$ is defined in (\ref{defn-L2-e}).
\end{lemma}

Assuming Lemma \ref{lemma-moving}, we prove Theorem \ref{THM: main} for BBM
type equations.

\begin{proof}
[Proof of Theorem \ref{spectrum-KDV} (BBM)]We prove (ii) and the proof of (i)
is similar. Assume that $n^{-}\left(  \mathcal{L}_{0}\right)  $ is odd and
$dP/dc<0$. Let $k_{1}^{-},\cdots,k_{l}^{-}$ be all the distinct negative
eigenvalues of $\mathcal{L}_{0}$. Choose $\delta>0$ small such that the
$l\ $disks $B\left(  k_{i}^{-};\delta\right)  $ are disjoint and still lie in
the left half plane. By Proposition \ref{Prop-asym-perturb}, there
exists$\ \lambda_{1}>0$ and $\delta$ small enough, such that for
$0<\lambda<\lambda_{1}$, $\mathcal{A}^{\lambda}$ has $n^{-}\left(
\mathcal{L}_{0}\right)  \ $eigenvalues (counting multiplicity) in $\cup
_{i=1}^{l}B\left(  k_{i}^{-};\delta\right)  .$ By Lemma \ref{lemma-moving}, if
$dP/dc<0$, then the zero eigenvalue of $\mathcal{A}^{0}$ is perturbed to a
positive eigenvalue $0<k_{\lambda}<\delta$ of $\mathcal{A}^{\lambda}$ for
small $\lambda$. Consider the region
\[
\Omega=\left\{  z\ |\ 0>\operatorname{Re}z>-2M\text{ and }\left\vert
\operatorname{Im}z\right\vert <2M\right\}  ,
\]
where $M$ is the uniform bound for $\left\Vert \mathcal{K}^{\lambda
}\right\Vert =\left\Vert \mathcal{A}^{\lambda}-\mathcal{M-}1\right\Vert $. We
claim that: for $\lambda$ small enough, $\mathcal{A}^{\lambda}$ has exactly
$n^{-}\left(  \mathcal{L}_{0}\right)  +1$ eigenvalues (with multiplicity) in
\[
\Omega_{\delta}=\left\{  z\ |\ 2\delta>\operatorname{Re}z>-2M\text{ and
}\left\vert \operatorname{Im}z\right\vert <2M\right\}  .
\]
That is, all eigenvalues of $\mathcal{A}^{\lambda}$ with real parts no greater
than $2\delta$ lie$\ $in $\cup_{i=1}^{l}B\left(  k_{i}^{-};\delta\right)  \cup
B\left(  0;\delta\right)  $. Suppose otherwise, there exists a sequence
$\lambda_{n}\rightarrow0$ and
\[
\left\{  u_{n}\right\}  \in H^{m}\left(  \mathbf{R}\right)  ,\ \ \ z_{n}%
\in\Omega/\left(  \cup_{i=1}^{l}B\left(  k_{i}^{-};\delta\right)  \cup
B\left(  0;\delta\right)  \right)
\]
such that $\left(  \mathcal{A}^{\lambda_{n}}-z_{n}\right)  u_{n}=0$. We
normalize $u_{n}$ by setting $\left\Vert u_{n}\right\Vert _{L_{e}^{2}}=1$.
Then by Lemma \ref{lemma-apriori}, we have $\left\Vert u_{n}\right\Vert
_{H^{\frac{m}{2}}}\leq C$. By the same argument as in the proof of Lemma
\ref{lemma-no-eigen-infy}, $u_{n}\rightarrow u_{\infty}\neq0$ weakly in
$H^{\frac{m}{2}}$. Let
\[
\lim_{n\rightarrow\infty}z_{n}=z_{\infty}\in\bar{\Omega}/\left(  \cup
_{i=1}^{l}B\left(  k_{i}^{-};\delta\right)  \cup B\left(  0;\delta\right)
\right)
\]
then $\mathcal{L}^{0}u_{\infty}=z_{\infty}u_{\infty}$, which is a
contradiction. The claim is proved and thus for $\lambda$ small enough,
$\mathcal{A}^{\lambda}$ has exactly $n^{-}\left(  \mathcal{L}_{0}\right)  $
eigenvalues in $\Omega$.

Suppose Theorem \ref{spectrum-KDV} (i) is not true, then $\mathcal{A}%
^{\lambda}$ has no kernel for any $\lambda>0$. Define $n_{\Omega}\left(
\lambda\right)  $ to be the number of eigenvalues (with multiplicity) of
$\mathcal{A}^{\lambda}\ $in $\Omega$. Since by \ref{bound-essential}, the
region $\Omega\ $is away from the essential spectrum of $\mathcal{A}^{\lambda
}$, $n_{\Omega}\left(  \lambda\right)  $ is always a finite integer. In the
above, we have proved that $n_{\Omega}\left(  \lambda\right)  =n^{-}\left(
\mathcal{L}_{0}\right)  $ is odd, for $\lambda$ small enough. By Lemma
\ref{lemma-no-eigen-infy}, there exists $\Lambda>0$ such that $n_{\Omega
}\left(  \lambda\right)  =0$ for $\lambda>\Lambda$. Define two sets
\[
S_{\text{odd}}=\left\{  \lambda>0|\text{ }n_{\Omega}\left(  \lambda\right)
\text{ is odd}\right\}  ,\ S_{\text{even}}=\left\{  \lambda>0|\text{
}n_{\Omega}\left(  \lambda\right)  \text{ is even}\right\}  .
\]
Then both sets are non-empty. Below, we show that both $S_{\text{odd}}$ and
$S_{\text{even}}$ are open sets. Let $\lambda_{0}\in$ $S_{\text{odd}}$ and
denote $k_{1},\cdots,k_{l}$ $\left(  l\leq n_{\Omega}\left(  \lambda
_{0}\right)  \right)  \ $to be all distinct eigenvalues of $\mathcal{A}%
^{\lambda_{0}}$ in $\Omega$. Denote $ih_{1},\cdots,ih_{p}$ to be all
eigenvalues of $\mathcal{A}^{\lambda_{0}}$ on the imaginary axis. Then
$\left\vert h_{j}\right\vert \leq M$, $1\leq j\leq p$. Choose $\delta>0$
sufficiently small such that the disks $B\left(  k_{i};\delta\right)  $
$\left(  1\leq i\leq l\right)  $ and $B\left(  ih_{j};\delta\right)  \ \left(
1\leq j\leq p\right)  $ are disjoint, $B\left(  k_{i};\delta\right)
\subset\Omega$ and $B\left(  ih_{j};\delta\right)  $ does not contain $0$.
Note that $\mathcal{A}^{\lambda}$ is analytic in $\lambda$ for $\lambda
\in\left(  0,+\infty\right)  $. By the analytic perturbation theory
(\cite{hislop-sig-book}), if $\left\vert \lambda-\lambda_{0}\right\vert $ is
sufficiently small, any eigenvalue of $\mathcal{A}^{\lambda}$ in
$\Omega_{\delta}$ lies in one of the disks $B\left(  k_{i};\delta\right)  $ or
$B\left(  ih_{j};\delta\right)  $. So $n_{\Omega}\left(  \lambda\right)  $ is
the number $n_{\Omega}\left(  \lambda_{0}\right)  $ plus the number of
eigenvalues in $\cup_{i=1}^{p}B\left(  ih_{j};\delta\right)  $ with the
negative real part. The second number is even, since the complex eigenvalues
of $\mathcal{A}^{\lambda}$ appears in conjugate pairs. Thus, $n_{\Omega
}\left(  \lambda\right)  $ is odd for $\left\vert \lambda-\lambda
_{0}\right\vert $ small enough. This shows that $S_{\text{odd}}$ is open.
Similarly, $S_{\text{even}}$ is open$.$ Thus, $\left(  0,+\infty\right)  $ is
the union of two non-empty, disjoint open sets $S_{\text{odd}}$ and
$S_{\text{even}}$. This is a contradiction.

So there exists $\lambda>0$ and $0\neq u\in$ $H^{m}\left(  \mathbf{R}\right)
$ such that $\mathcal{A}^{\lambda}u=0$. Then $e^{\lambda t}u\left(  x\right)
$ is purely growing mode solution to (\ref{L-bbm}). One could also get more
regularity of $u\left(  x\right)  $, as in the usual proof of the regularity
of solitary waves (i.e. \cite{benjamin90}).
\end{proof}

It remains to prove the moving kernel formula (\ref{formula-moving}).

\begin{proof}
[Proof of Lemma \ref{lemma-moving}]We use $C$ to denote a generic constant in
our estimates below. As described at the beginning of this subsection, for
$\lambda>0$ small enough, there exists $u_{\lambda}\in H^{m}\left(
\mathbf{R}\right)  $, such that $\left(  \mathcal{A}^{\lambda}-k_{\lambda
}\right)  u_{\lambda}=0$ with $k_{\lambda}\in\mathbf{R}$ and $\lim
_{\lambda\rightarrow0+}k_{\lambda}=0$. We normalize $u_{\lambda}$ by setting
$\left\Vert u_{\lambda}\right\Vert _{L_{e}^{2}}=1$. Then by Lemma
\ref{lemma-apriori}, we have $\left\Vert u_{\lambda}\right\Vert _{H^{\frac
{m}{2}}}\leq C$ and as in the proof of Lemma \ref{lemma-no-eigen-infy},
$u_{\lambda}\rightarrow u_{0}\neq0$ weakly in $H^{\frac{m}{2}}$. Since
$\mathcal{A}^{0}u_{0}=\mathcal{L}_{0}u_{0}=0$ and $\ker\mathcal{L}_{0}=$
$\left\{  u_{cx}\right\}  $, we have $u_{0}=c_{0}u_{cx}$ for some $c_{0}\neq
0$. Moreover, we have $\left\Vert u_{\lambda}-u_{0}\right\Vert _{H^{\frac
{m}{2}}}=0$. To show that, first we note that $\left\Vert u_{\lambda}%
-u_{0}\right\Vert _{L_{e}^{2}}\rightarrow0$, since
\[
\left\Vert u_{\lambda}-u_{0}\right\Vert _{L_{e}^{2}}^{2}\leq\int_{\left\vert
x\right\vert \leq R}e\left(  x\right)  \left\vert u_{\lambda}-u_{0}\right\vert
^{2}\ dx+\max_{\left\vert x\right\vert \geq R}e\left(  x\right)  \left\Vert
u_{\lambda}-u_{0}\right\Vert _{L^{2}}^{2},
\]
and the second term is arbitrarily small for large $R$ while the first term
tends to zero by the local compactness. Since
\[
\left(  \mathcal{A}^{\lambda}-k_{\lambda}\right)  \left(  u_{\lambda}%
-u_{0}\right)  =k_{\lambda}u_{0}+\left(  \mathcal{A}^{0}-\mathcal{A}^{\lambda
}\right)  u_{0},
\]
by Lemma \ref{lemma-apriori} we have
\[
\left\Vert u_{\lambda}-u_{0}\right\Vert _{H^{\frac{m}{2}}}\leq C\left(
\left\Vert u_{\lambda}-u_{0}\right\Vert _{L_{e}^{2}}^{2}+\left\vert
k_{\lambda}\right\vert \left\Vert u_{0}\right\Vert _{L^{2}}^{2}+\left\Vert
\left(  \mathcal{A}^{0}-\mathcal{A}^{\lambda}\right)  u_{0}\right\Vert
_{L^{2}}^{2}\right)  \rightarrow0,
\]
when $\lambda\rightarrow0+$. We can assume $c_{0}=1$ by renormalizing the sequence.

Next, we show that $\lim_{\lambda\rightarrow0+}\frac{k_{\lambda}}{\lambda}=0$.
From $\left(  \mathcal{A}^{\lambda}-k_{\lambda}\right)  u_{\lambda}=0,$ we
have
\begin{equation}
\frac{k_{\lambda}}{\lambda}u_{\lambda}=\mathcal{A}^{0}\frac{u_{\lambda}%
}{\lambda}+\frac{\mathcal{A}^{\lambda}-\mathcal{A}^{0}}{\lambda}u_{\lambda}.
\label{eqn-1st}%
\end{equation}
Taking the inner product of above with $u_{cx}$, we get
\[
\frac{k_{\lambda}}{\lambda}\left(  u_{\lambda},u_{cx}\right)  =\left(
\frac{\mathcal{A}^{\lambda}-\mathcal{A}^{0}}{\lambda}u_{\lambda}%
,u_{cx}\right)  :=m\left(  \lambda\right)  .
\]
We have
\begin{align*}
m\left(  \lambda\right)   &  =\left(  \frac{1}{c}\frac{1}{\lambda-\mathcal{D}%
}\left(  1+f^{\prime}\left(  u_{c}\right)  \right)  u_{\lambda},u_{cx}\right)
=\frac{1}{c}\left(  \left(  1+f^{\prime}\left(  u_{c}\right)  \right)
u_{\lambda},\frac{1}{\lambda+\mathcal{D}}u_{cx}\right) \\
&  =\frac{1}{c^{2}}\left(  \left(  1+f^{\prime}\left(  u_{c}\right)  \right)
u_{\lambda},\left(  1-\mathcal{E}^{\lambda,+}\right)  u_{c}\right)  \text{
}\rightarrow\frac{1}{c^{2}}\left(  \left(  1+f^{\prime}\left(  u_{c}\right)
\right)  u_{cx},u_{c}\right) \\
&  =\frac{1}{c^{2}}\int\frac{d}{dx}\left(  \frac{1}{2}u_{c}^{2}+F\left(
u_{c}\right)  \right)  \ dx=0,
\end{align*}
where $F\left(  u\right)  =\int_{0}^{u}f^{\prime}\left(  s\right)  sds$ and in
the above $\lim_{\lambda\rightarrow0+}\mathcal{E}^{\lambda,+}=0$ is used. So
\[
\lim_{\lambda\rightarrow0+}\frac{k_{\lambda}}{\lambda}=\lim_{\lambda
\rightarrow0+}\frac{m\left(  \lambda\right)  }{\left(  u_{\lambda}%
,u_{cx}\right)  }=0.
\]
We write $u_{\lambda}=c_{\lambda}u_{cx}+\lambda v_{\lambda}$, where
$c_{\lambda}=\left(  u_{\lambda},u_{cx}\right)  /\left(  u_{cx},u_{cx}\right)
$. Then $\left(  v_{\lambda},u_{cx}\right)  =0$ and $c_{\lambda}\rightarrow1$
when $\lambda\rightarrow0+$. We claim that: $\left\Vert v_{\lambda}\right\Vert
_{L_{e}^{2}}\leq C$ (independent of $\lambda$). Suppose otherwise, there
exists a sequence $\lambda_{n}\rightarrow0+$ such that $\left\Vert
v_{\lambda_{n}}\right\Vert _{L_{e}^{2}}\geq n$. Denote $\tilde{v}_{\lambda
_{n}}=v_{\lambda_{n}}/\left\Vert v_{\lambda_{n}}\right\Vert _{L_{e}^{2}}$.
Then $\left\Vert \tilde{v}_{\lambda_{n}}\right\Vert _{L_{e}^{2}}=1$ and
$\tilde{v}_{\lambda_{n}}$ satisfies the equation
\begin{equation}
\mathcal{A}^{\lambda_{n}}\tilde{v}_{\lambda_{n}}=\frac{1}{\left\Vert \tilde
{v}_{\lambda_{n}}\right\Vert _{L_{e}^{2}}}\left(  \frac{k_{\lambda_{n}}%
}{\lambda_{n}}u_{\lambda_{n}}-c_{\lambda_{n}}\frac{\mathcal{A}^{\lambda_{n}%
}-\mathcal{A}^{0}}{\lambda_{n}}u_{cx}\right)  . \label{eqn-vn-tilt}%
\end{equation}
Denote%
\[
w_{\lambda}\left(  x\right)  =\frac{\mathcal{A}^{\lambda}-\mathcal{A}^{0}%
}{\lambda}u_{cx},
\]
then
\begin{align*}
w_{\lambda}\left(  x\right)   &  =\frac{1}{c}\frac{1}{\lambda-\mathcal{D}%
}\left(  1+f^{\prime}\left(  u_{c}\right)  \right)  u_{cx}=\frac{1}%
{\lambda-\mathcal{D}}\frac{d}{dx}\left(  \left(  \mathcal{M}+1\right)
u_{c}\right) \\
&  =\frac{1}{c}\frac{\mathcal{D}}{\lambda-\mathcal{D}}\left(  \mathcal{M}%
+1\right)  u_{c}=\frac{1}{c}\left(  \mathcal{E}^{\lambda,-}-1\right)  \left(
\mathcal{M}+1\right)  u_{c},
\end{align*}
where we use the equation
\[
\mathcal{L}_{0}u_{cx}=\mathcal{M}u_{cx}+\left(  1-\frac{1}{c}\right)
u_{cx}-\frac{1}{c}f^{\prime}\left(  u_{c}\right)  u_{cx}=0.
\]
By Lemma \ref{lemma-e-lb}, $\left\Vert w_{\lambda}\right\Vert _{L^{2}}\leq C$
(independent of $\lambda$), and
\begin{equation}
w_{\lambda}\left(  x\right)  \rightarrow-\frac{1}{c}\left(  \mathcal{M}%
+1\right)  u_{c}=\frac{1}{c^{2}}\left(  u_{c}+f\left(  u_{c}\right)  \right)
\label{limit-w-lb}%
\end{equation}
strongly in $L^{2},\ $when $\lambda\rightarrow0+$. So by Lemma
\ref{lemma-apriori}, we have $\left\Vert \tilde{v}_{\lambda_{n}}\right\Vert
_{H^{\frac{m}{2}}}\leq C$. Then, as before, $\tilde{v}_{\lambda_{n}%
}\rightarrow$ $\tilde{v}_{0}\neq0$ weakly in $H^{\frac{m}{2}}$. Since
$\frac{k_{\lambda_{n}}}{\lambda_{n}},\frac{1}{\left\Vert \tilde{v}%
_{\lambda_{n}}\right\Vert _{L_{e}^{2}}}\rightarrow0$, we have $\mathcal{A}%
^{0}\tilde{v}_{0}=0$. So $\tilde{v}_{0}=c_{1}u_{cx}$ for some $c_{1}\neq0$.
But since $\left(  \tilde{v}_{\lambda_{n}},u_{cx}\right)  =0$, we have
$\left(  \tilde{v}_{0},u_{cx}\right)  =0,\ $a contradiction. This establishes
the uniform bound for $\left\Vert v_{\lambda}\right\Vert _{L_{e}^{2}}$. The
equation satisfied by $v_{\lambda}$ is
\[
\mathcal{A}^{\lambda}v_{\lambda}=\frac{k_{\lambda}}{\lambda_{n}}u_{\lambda
}-c_{\lambda}\frac{\mathcal{A}^{\lambda}-\mathcal{A}^{0}}{\lambda}u_{cx}%
=\frac{k_{\lambda}}{\lambda_{n}}u_{\lambda}-c_{\lambda}w_{\lambda}.
\]
Applying Lemma \ref{lemma-apriori} to the above equation, we have $\left\Vert
v_{\lambda}\right\Vert _{H^{\frac{m}{2}}}\leq C$ and thus $v_{\lambda
}\rightarrow$ $v_{0}$ weakly in $H^{\frac{m}{2}}$. By (\ref{limit-w-lb}),
$v_{0}$ satisfies
\[
\mathcal{A}^{0}v_{0}=\mathcal{L}_{0}v_{0}=\frac{1}{c}\left(  \mathcal{M}%
+1\right)  u_{c}.
\]
Taking $\partial_{c}$ of (\ref{solitary-bbm}), we have
\begin{equation}
\mathcal{L}_{0}\partial_{c}u_{c}=-\frac{1}{c}\left(  \mathcal{M}+1\right)
u_{c}. \label{eqn-dc-uc}%
\end{equation}
Thus $\mathcal{L}_{0}\left(  v_{0}+\partial_{c}u_{c}\right)  =0$. Since
$\left(  v_{0},u_{cx}\right)  =\lim_{\lambda\rightarrow0+}\left(  v_{\lambda
},u_{cx}\right)  =0,$ we have
\[
v_{0}=-\partial_{c}u_{c}+d_{0}u_{cx},\ d_{0}=\left(  \partial_{c}u_{c}%
,u_{cx}\right)  \ /\left\Vert u_{cx}\right\Vert _{L^{2}}^{2}.
\]
Similar to the proof of $\left\Vert u_{\lambda}-u_{0}\right\Vert _{H^{\frac
{m}{2}}}\rightarrow0$, we have $\left\Vert v_{\lambda}-v_{0}\right\Vert
_{H^{\frac{m}{2}}}\rightarrow0.$ We rewrite
\[
u_{\lambda}=c_{\lambda}u_{cx}+\lambda v_{\lambda}=\bar{c}_{\lambda}%
u_{cx}+\lambda\bar{v}_{\lambda},
\]
where $\bar{c}_{\lambda}=c_{\lambda}+\lambda d_{0}$, $\bar{v}_{\lambda
}=v_{\lambda}-d_{0}u_{cx}$. Then $\bar{c}_{\lambda}\rightarrow1$, $\bar
{v}_{\lambda}\rightarrow-\partial_{c}u_{c},$when $\lambda\rightarrow0+$.

Now we compute $\lim_{\lambda\rightarrow0+}\frac{k_{\lambda}}{\lambda^{2}}$.
From (\ref{eqn-1st}), we have
\[
\mathcal{A}^{0}\frac{u_{\lambda}}{\lambda^{2}}+\frac{\mathcal{A}^{\lambda
}-\mathcal{A}^{0}}{\lambda}\left(  \frac{\bar{c}_{\lambda}}{\lambda}%
u_{cx}+\bar{v}_{\lambda}\right)  =\frac{k_{\lambda}}{\lambda^{2}}u_{\lambda}.
\]
Taking the inner product of above with $u_{cx}$, we have%
\[
\frac{k_{\lambda}}{\lambda^{2}}\left(  u_{\lambda},u_{cx}\right)  =\bar
{c}_{\lambda}\left(  \frac{\mathcal{A}^{\lambda}-\mathcal{A}^{0}}{\lambda^{2}%
}u_{cx},u_{cx}\right)  +\left(  \frac{\mathcal{A}^{\lambda}-\mathcal{A}^{0}%
}{\lambda}\bar{v}_{\lambda},u_{cx}\right)  =\bar{c}_{\lambda}I_{1}+I_{2}.
\]
For the first term, we have
\begin{align*}
I_{1}  &  =\left(  \frac{\mathcal{A}^{\lambda}-\mathcal{A}^{0}}{\lambda^{2}%
}u_{cx},u_{cx}\right)  =\left(  \frac{w_{\lambda}\left(  x\right)  }{\lambda
},u_{cx}\right)  =\frac{1}{c}\left(  \frac{\mathcal{D}}{\left(  \lambda
-\mathcal{D}\right)  \lambda}\left(  \mathcal{M}+1\right)  u_{c},u_{cx}\right)
\\
&  =\frac{1}{c}\left(  \frac{1}{\lambda-\mathcal{D}}\left(  \mathcal{M}%
+1\right)  u_{c},u_{cx}\right)  -\frac{1}{c\lambda}\left(  \left(
\mathcal{M}+1\right)  u_{c},u_{cx}\right) \\
&  =-\frac{1}{c^{2}}\left(  \left(  \mathcal{E}^{\lambda,-}-1\right)  \left(
\mathcal{M}+1\right)  u_{c},u_{c}\right)  -\frac{1}{c^{2}\lambda}\left(
u_{c}+f\left(  u_{c}\right)  ,u_{cx}\right) \\
&  =-\frac{1}{c^{2}}\left(  \left(  \mathcal{E}^{\lambda,-}-1\right)  \left(
\mathcal{M}+1\right)  u_{c},u_{c}\right)  \rightarrow\frac{1}{c^{2}}\left(
\left(  \mathcal{M}+1\right)  u_{c},u_{c}\right)  ,\ \text{when\ }%
\lambda\rightarrow0+.
\end{align*}
For the second term, we have
\begin{align*}
I_{2}  &  =\left(  \frac{\mathcal{A}^{\lambda}-\mathcal{A}^{0}}{\lambda}%
\bar{v}_{\lambda},u_{cx}\right)  =\frac{1}{c}\left(  \frac{1}{\lambda
-\mathcal{D}}\left(  1+f^{\prime}\left(  u_{c}\right)  \right)  \bar
{v}_{\lambda},u_{cx}\right) \\
&  =-\frac{1}{c^{2}}\left(  \frac{\mathcal{D}}{\lambda-\mathcal{D}}\left(
1+f^{\prime}\left(  u_{c}\right)  \right)  \bar{v}_{\lambda},u_{c}\right)
=-\frac{1}{c^{2}}\left(  \left(  \mathcal{E}^{\lambda,-}-1\right)  \left(
1+f^{\prime}\left(  u_{c}\right)  \right)  \bar{v}_{\lambda},u_{c}\right) \\
&  \rightarrow-\frac{1}{c^{2}}\left(  \left(  1+f^{\prime}\left(
u_{c}\right)  \right)  \partial_{c}u_{c},u_{c}\right)  ,\ \text{when }%
\lambda\rightarrow0\text{. }%
\end{align*}
Thus
\begin{align*}
\lim_{\lambda\rightarrow0+}\frac{k_{\lambda}}{\lambda^{2}}  &  =\lim
_{\lambda\rightarrow0+}\frac{\bar{c}_{\lambda}I_{1}+I_{2}}{\left(  u_{\lambda
},u_{cx}\right)  }\\
&  =\left[  \frac{1}{c^{2}}\left(  \left(  \mathcal{M}+1\right)  u_{c}%
,u_{c}\right)  -\frac{1}{c^{2}}\left(  \left(  1+f^{\prime}\left(
u_{c}\right)  \right)  \partial_{c}u_{c},u_{c}\right)  \right]  /\left\Vert
u_{cx}\right\Vert _{L^{2}}^{2}\\
&  =-\frac{1}{c}\left(  \left(  \mathcal{M}+1\right)  \partial_{c}u_{c}%
,u_{c}\right)  =-\frac{1}{c}\frac{dP}{dc},
\end{align*}
since by (\ref{eqn-dc-uc})%
\[
\left(  \mathcal{M}+1\right)  u_{c}-\left(  1+f^{\prime}\left(  u_{c}\right)
\right)  \partial_{c}u_{c}=-c\left(  \mathcal{M}+1\right)  \partial_{c}u_{c}.
\]

\end{proof}

\section{Regularized Boussinesq type}

Consider a solitary wave $u\left(  x,t\right)  =u_{c}\left(  x-ct\right)
\ \left(  c^{2}>1\right)  $ of the regularized Boussinesq (RBou) type equation
(\ref{RBOU}). Then $u_{c}$ satisfies the equation%
\begin{equation}
\mathcal{M}u_{c}+\left(  1-\frac{1}{c^{2}}\right)  u_{c}-\frac{1}{c^{2}%
}f\left(  u_{c}\right)  =0\text{.} \label{solitary-RBou}%
\end{equation}
The linearized equation in the traveling frame $\left(  x-ct,t\right)  $ is
\begin{equation}
\left(  \partial_{t}-c\partial_{x}\right)  ^{2}\left(  u+\mathcal{M}u\right)
-\partial_{x}^{2}\left(  u+f^{\prime}\left(  u_{c}\right)  u\right)  =0.
\label{L-RBou}%
\end{equation}
For a growing mode $e^{\lambda t}u\left(  x\right)  $ $\left(
\operatorname{Re}\lambda>0\right)  $, $u\left(  x\right)  $ satisfies
\begin{equation}
\left(  \lambda-c\partial_{x}\right)  ^{2}\left(  u+\mathcal{M}u\right)
-\partial_{x}^{2}\left(  u+f^{\prime}\left(  u_{c}\right)  u\right)  =0.
\label{spectral-Rbou}%
\end{equation}
So we define the following dispersion operator $\mathcal{A}^{\lambda}%
:H^{m}\rightarrow L^{2}$ $\left(  \lambda>0\right)  $
\[
\mathcal{A}^{\lambda}u=\mathcal{M}u+u-\left(  \frac{\partial_{x}}%
{\lambda-c\partial_{x}}\right)  ^{2}\left(  u+f^{\prime}\left(  u_{c}\right)
u\right)
\]
and the existence of a purely growing mode is reduced to find $\lambda$ $>0$
such that $\mathcal{A}^{\lambda}$ has a nontrivial kernel. Since when
$\lambda\rightarrow0+,$
\begin{equation}
\frac{\partial_{x}}{\lambda-c\partial_{x}}=\frac{\mathcal{D}}{\lambda
-\mathcal{D}}=\frac{1}{c}\left(  \mathcal{E}^{\lambda,-}-1\right)
\rightarrow-\frac{1}{c}\text{ strongly in }L^{2}\text{,}
\label{convergence-frac-dx}%
\end{equation}
the zero limit of the operator$\mathcal{A}^{\lambda}$ is
\begin{equation}
\mathcal{L}_{0}:=\mathcal{M}+\left(  1-\frac{1}{c^{2}}\right)  -\frac{1}%
{c^{2}}f^{\prime}\left(  u_{c}\right)  . \label{L0-RBou}%
\end{equation}

The proof of Theorem \ref{THM: main} for RBou case is very similar to the BBM
case, so we only give a sketch of the proof of the moving kernel formula.

\begin{lemma}
\label{lemma-moving-RBou}Assume $\ker\mathcal{L}^{0}=$ $\left\{
u_{cx}\right\}  $. For $\lambda>0$ small enough, let $k_{\lambda}\in
\mathbf{R}$ to be the only eigenvalue of $\mathcal{A}^{\lambda}$ near zero.
Then we have$\ $
\[
\lim_{\lambda\rightarrow0+}\frac{k_{\lambda}}{\lambda^{2}}=-\frac{1}{c^{2}%
}\frac{dP}{dc}/\left\Vert u_{cx}\right\Vert _{L^{2}}^{2},
\]
where
\begin{equation}
P\left(  c\right)  =c\left(  \left(  \mathcal{M+}1\right)  u_{c},u_{c}\right)
. \label{P-RBou}%
\end{equation}

\end{lemma}

\begin{proof}
For $\lambda>0$ small enough, let
\[
u_{\lambda}\in H^{m}\left(  \mathbf{R}\right)  ,\ k_{\lambda}\in
\mathbf{R},\ \lim_{\lambda\rightarrow0+}k_{\lambda}=0,
\]
such that $\left(  \mathcal{A}^{\lambda}-k_{\lambda}\right)  u_{\lambda}=0$.
We normalize $u_{\lambda}$ by setting $\left\Vert u_{\lambda}\right\Vert
_{L_{e}^{2}}=1$. Then as in the BBM case, we have $\left\Vert u_{\lambda
}\right\Vert _{H^{\frac{m}{2}}}\leq C$ and $u_{\lambda}\rightarrow u_{cx}\ $in
$H^{\frac{m}{2}}$ by a renormalization, under our assumption that
$\ker\mathcal{L}_{0}=$ $\left\{  u_{cx}\right\}  $.

First, we show that $\lim_{\lambda\rightarrow0+}\frac{k_{\lambda}}{\lambda}%
=0$. As in the BBM case, we have
\[
\frac{k_{\lambda}}{\lambda}\left(  u_{\lambda},u_{cx}\right)  =\left(
\frac{\mathcal{A}^{\lambda}-\mathcal{A}^{0}}{\lambda}u_{\lambda}%
,u_{cx}\right)  :=m\left(  \lambda\right)  ,
\]
where
\[
\frac{\mathcal{A}^{\lambda}-\mathcal{A}^{0}}{\lambda}=\frac{1}{c^{2}}\left(
\frac{2}{\lambda-\mathcal{D}}-\frac{\lambda}{\left(  \lambda-\mathcal{D}%
\right)  ^{2}}\right)  \left(  1+f^{\prime}\left(  u_{c}\right)  \right)  .
\]
We have
\begin{align*}
m\left(  \lambda\right)   &  =\frac{1}{c^{2}}\left(  \left[  \frac{2}%
{\lambda-\mathcal{D}}-\frac{\lambda}{\left(  \lambda-\mathcal{D}\right)  ^{2}%
}\right]  \left(  1+f^{\prime}\left(  u_{c}\right)  \right)  u_{\lambda
},u_{cx}\right) \\
&  =\frac{1}{c^{2}}\left(  \left(  1+f^{\prime}\left(  u_{c}\right)  \right)
u_{\lambda},\left(  \frac{2}{\lambda+\mathcal{D}}-\frac{\lambda}{\left(
\lambda+\mathcal{D}\right)  ^{2}}\right)  u_{cx}\right) \\
&  =\frac{1}{c^{3}}\left(  \left(  1+f^{\prime}\left(  u_{c}\right)  \right)
u_{\lambda},\left(  \frac{2\mathcal{D}}{\lambda+\mathcal{D}}-\frac
{\lambda\mathcal{D}}{\left(  \lambda+\mathcal{D}\right)  ^{2}}\right)
u_{c}\right) \\
&  =\frac{1}{c^{3}}\left(  \left(  1+f^{\prime}\left(  u_{c}\right)  \right)
u_{\lambda},\left(  1-\mathcal{E}^{\lambda,+}\right)  \left(  2-\mathcal{E}%
^{\lambda,+}\right)  u_{c}\right) \\
&  \rightarrow\frac{2}{c^{3}}\left(  \left(  1+f^{\prime}\left(  u_{c}\right)
\right)  u_{cx},u_{c}\right)  =0\text{,}%
\end{align*}
and thus
\[
\lim_{\lambda\rightarrow0+}\frac{k_{\lambda}}{\lambda}=\lim_{\lambda
\rightarrow0+}\frac{m\left(  \lambda\right)  }{\left(  u_{\lambda}%
,u_{cx}\right)  }=0.
\]
Similarly to the BBM case, we can show that $u_{\lambda}=\bar{c}_{\lambda
}u_{cx}+\lambda\bar{v}_{\lambda}$, with $\bar{c}_{\lambda}\rightarrow1$,
$\bar{v}_{\lambda}\rightarrow-\partial_{c}u_{c}$ in $H^{\frac{m}{2}},\ $when
$\lambda\rightarrow0+$. In the proof, we use the facts that%
\begin{align*}
w_{\lambda}\left(  x\right)   &  =\frac{\mathcal{A}^{\lambda}-\mathcal{A}^{0}%
}{\lambda}u_{cx}=\frac{1}{c^{2}}\left(  \frac{2}{\lambda-\mathcal{D}}%
-\frac{\lambda}{\left(  \lambda-\mathcal{D}\right)  ^{2}}\right)  \left(
1+f^{\prime}\left(  u_{c}\right)  \right)  u_{cx}\\
&  =\left(  \frac{2}{\lambda-\mathcal{D}}-\frac{\lambda}{\left(
\lambda-\mathcal{D}\right)  ^{2}}\right)  \left(  \mathcal{M}+1\right)
u_{cx}=\frac{1}{c}\left(  \frac{2\mathcal{D}}{\lambda-\mathcal{D}}%
-\frac{\lambda\mathcal{D}}{\left(  \lambda-\mathcal{D}\right)  ^{2}}\right)
\left(  \mathcal{M}+1\right)  u_{c}\\
&  =\frac{1}{c}\left(  \mathcal{E}^{\lambda,-}-1\right)  \left(
2-\mathcal{E}^{\lambda,-}\right)  \left(  \mathcal{M}+1\right)  u_{c}%
\rightarrow-\frac{2}{c}\left(  \mathcal{M}+1\right)  u_{c}\text{, when
}\lambda\rightarrow0+\text{. }%
\end{align*}
and
\begin{equation}
\mathcal{L}_{0}\partial_{c}u_{c}=-\frac{2}{c^{3}}\left(  u_{c}+f\left(
u_{c}\right)  \right)  =-\frac{2}{c}\left(  \mathcal{M}+1\right)  u_{c}.
\label{eqn-dc-uc-RBou}%
\end{equation}
Next, we compute $\lim_{\lambda\rightarrow0+}\frac{k_{\lambda}}{\lambda^{2}}$
by using
\[
\frac{k_{\lambda}}{\lambda^{2}}\left(  u_{\lambda},u_{cx}\right)  =\bar
{c}_{\lambda}\left(  \frac{\mathcal{A}^{\lambda}-\mathcal{A}^{0}}{\lambda^{2}%
}u_{cx},u_{cx}\right)  +\left(  \frac{\mathcal{A}^{\lambda}-\mathcal{A}^{0}%
}{\lambda}\bar{v}_{\lambda},u_{cx}\right)  =\bar{c}_{\lambda}I_{1}+I_{2}.
\]
For the first term, we have
\begin{align*}
I_{1}  &  =\left(  \frac{\mathcal{A}^{\lambda}-\mathcal{A}^{0}}{\lambda^{2}%
}u_{cx},u_{cx}\right)  =\left(  \frac{w_{\lambda}\left(  x\right)  }{\lambda
},u_{cx}\right) \\
&  =\frac{1}{c}\left(  \left[  \frac{2\mathcal{D}}{\left(  \lambda
-\mathcal{D}\right)  \lambda}-\frac{\mathcal{D}}{\left(  \lambda
-\mathcal{D}\right)  ^{2}}\right]  \left(  \mathcal{M}+1\right)  u_{c}%
,u_{cx}\right) \\
&  =-\frac{1}{c^{2}}\left(  \left[  \frac{2\mathcal{D}^{2}}{\left(
\lambda-\mathcal{D}\right)  \lambda}-\frac{\mathcal{D}^{2}}{\left(
\lambda-\mathcal{D}\right)  ^{2}}\right]  \left(  \mathcal{M}+1\right)
u_{c},u_{c}\right) \\
&  =-\frac{2}{c^{2}}\left(  \left(  \mathcal{E}^{\lambda,-}-1\right)  \left(
\mathcal{M}+1\right)  u_{c},u_{c}\right)  +\frac{1}{c^{2}\lambda}\left(
\mathcal{D}\left(  \mathcal{M}+1\right)  u_{c},u_{c}\right)  \ \\
&  \ \ \ \ \ \ \ \ \ \ \ \ \ \ \ +\frac{1}{c^{2}}\left(  \left(
\mathcal{E}^{\lambda,-}-1\right)  ^{2}\left(  \mathcal{M}+1\right)
u_{c},u_{c}\right) \\
&  \rightarrow\frac{3}{c^{2}}\left(  \left(  \mathcal{M}+1\right)  u_{c}%
,u_{c}\right)  \text{, when }\lambda\rightarrow0+,
\end{align*}
since $\mathcal{E}^{\lambda,-}\rightarrow0$ and
\[
\left(  \mathcal{D}\left(  \mathcal{M}+1\right)  u_{c},u_{c}\right)  =c\left(
u_{cx},\left(  \mathcal{M}+1\right)  u_{c}\right)  =\frac{1}{c}\left(
u_{cx},u_{c}+f\left(  u_{c}\right)  \right)  =0.
\]
For the second term, we have
\begin{align*}
I_{2}  &  =\left(  \frac{\mathcal{A}^{\lambda}-\mathcal{A}^{0}}{\lambda}%
\bar{v}_{\lambda},u_{cx}\right)  =\frac{1}{c^{2}}\left(  \left(  \frac
{2}{\lambda-\mathcal{D}}-\frac{\lambda}{\left(  \lambda-\mathcal{D}\right)
^{2}}\right)  \left(  1+f^{\prime}\left(  u_{c}\right)  \right)  \bar
{v}_{\lambda},u_{cx}\right) \\
&  =-\frac{1}{c^{3}}\left(  \left(  \frac{2\mathcal{D}}{\lambda-\mathcal{D}%
}-\frac{\lambda\mathcal{D}}{\left(  \lambda-\mathcal{D}\right)  ^{2}}\right)
\left(  1+f^{\prime}\left(  u_{c}\right)  \right)  \bar{v}_{\lambda}%
,u_{c}\right) \\
&  =-\frac{1}{c^{3}}\left(  \left(  \mathcal{E}^{\lambda,-}-1\right)  \left(
2-\mathcal{E}^{\lambda,-}\right)  \left(  1+f^{\prime}\left(  u_{c}\right)
\right)  \bar{v}_{\lambda},u_{c}\right) \\
&  \rightarrow-\frac{2}{c^{3}}\left(  \left(  1+f^{\prime}\left(
u_{c}\right)  \right)  \partial_{c}u_{c},u_{c}\right)  ,\ \text{when }%
\lambda\rightarrow0+\text{. }%
\end{align*}
Thus
\begin{align*}
\lim_{\lambda\rightarrow0+}\frac{k_{\lambda}}{\lambda^{2}}  &  =\lim
_{\lambda\rightarrow0+}\frac{\bar{c}_{\lambda}I_{1}+I_{2}}{\left(  u_{\lambda
},u_{cx}\right)  }\\
&  =\left[  \frac{3}{c^{2}}\left(  \left(  \mathcal{M}+1\right)  u_{c}%
,u_{c}\right)  -\frac{2}{c^{3}}\left(  \left(  1+f^{\prime}\left(
u_{c}\right)  \right)  \partial_{c}u_{c},u_{c}\right)  \right]  /\left\Vert
u_{cx}\right\Vert _{L^{2}}^{2}\\
&  =\left[  -\frac{1}{c^{2}}\left(  \left(  \mathcal{M}+1\right)  u_{c}%
,u_{c}\right)  -\frac{2}{c}\left(  \left(  \mathcal{M}+1\right)  \partial
_{c}u_{c},u_{c}\right)  \right]  /\left\Vert u_{cx}\right\Vert _{L^{2}}^{2}\\
&  =-\frac{1}{c^{2}}\frac{dP}{dc}/\left\Vert u_{cx}\right\Vert _{L^{2}}^{2},
\end{align*}
since by (\ref{eqn-dc-uc-RBou})%
\[
\left(  1+f^{\prime}\left(  u_{c}\right)  \right)  \partial_{c}u_{c}%
=c^{2}\left(  \mathcal{M}+1\right)  \partial_{c}u_{c}+2c\left(  \mathcal{M}%
+1\right)  u_{c}.
\]

\end{proof}

As a corollary of the above proof, we show Theorem \ref{thm-transition} for
the RBou case. We skip the proof of Theorem \ref{thm-transition} for the BBM
and KDV cases, since they are very similar. Theorem \ref{thm-transition}
(RBou) follows from the next lemma.

\begin{lemma}
Assume $\ker\mathcal{L}_{0}=\left\{  u_{cx}\right\}  $. If there is a sequence
of purely growing modes $e^{\lambda_{n}t}u_{n}\left(  x\right)  $ $\left(
\lambda_{n}>0\right)  \ $for solitary waves $u_{c_{n}}\ $of (\ref{RBOU}), with
$\lambda_{n}\rightarrow0+$, $c_{n}\rightarrow c_{0}$, then we must have
$P^{\prime}\left(  c_{0}\right)  =0$.
\end{lemma}

\begin{proof}
The proof is almost the same as that of Lemma \ref{lemma-moving-RBou}, so we
only sketch it. The only difference is that now the computations depend on the
parameter $c_{n}$. Denote $\mathcal{E}_{n}^{\pm}=\frac{\mathcal{\lambda}_{n}%
}{\lambda_{n}\pm c_{n}\partial_{x}}$, then by the same argument as in the
proof of Lemma \ref{lemma-e-lb}, we have $s-\lim_{n\rightarrow\infty
}\mathcal{E}_{n}^{\pm}=0$. Then the operator
\[
\mathcal{A}^{\lambda_{n},c_{n}}=\mathcal{M}+1-\left(  \frac{\partial_{x}%
}{\lambda_{n}-c_{n}\partial_{x}}\right)  ^{2}\left(  1+f^{\prime}\left(
u_{c_{n}}\right)  \right)
\]
converges to
\[
\mathcal{L}_{0}:=\mathcal{M}+\left(  1-\frac{1}{c_{0}^{2}}\right)  -\frac
{1}{c_{0}^{2}}f^{\prime}\left(  u_{c_{0}}\right)
\]
strongly in $L^{2}$. We have $\mathcal{A}^{\lambda_{n},c_{n}}u_{n}=0$ and we
normalize $u_{n}\ $by$\ \left\Vert u_{n}\right\Vert _{L_{e_{n}}^{2}}=1$, where
$e_{n}=\left\vert f^{\prime}\left(  u_{c_{n}}\right)  \right\vert ^{2}$. As
before, it can be shown that $\left\Vert u_{n}\right\Vert _{H^{\frac{m}{2}}%
}\leq C$ (independent of $n$) and $u_{n}\rightarrow u_{c_{0}x}$ in
$H^{\frac{m}{2}}.$ Moreover, we have $u_{n}=\bar{c}_{n}u_{c_{n}x}+\lambda
_{n}\bar{v}_{n}$, where $\bar{c}_{n}\rightarrow1$ and $\bar{v}_{n}%
\rightarrow-\partial_{c}u_{c}|_{c_{0}}$ in $H^{\frac{m}{2}}$. From
$\mathcal{A}^{\lambda_{n},c_{n}}u_{n}=0$, it follows that
\[
0=\bar{c}_{n}\left(  \frac{\mathcal{A}^{\lambda_{n},c_{n}}-\mathcal{A}%
^{0,c_{n}}}{\lambda_{n}^{2}}u_{c_{n}x},u_{c_{n}x}\right)  +\left(
\frac{\mathcal{A}^{\lambda_{n},c_{n}}-\mathcal{A}^{0,c_{n}}}{\lambda_{n}}%
\bar{v}_{n},u_{c_{n}x}\right)  =\bar{c}_{n}I_{1}+I_{2},
\]
where$\ $%
\[
\mathcal{A}^{0,c_{n}}=\mathcal{M}+\left(  1-\frac{1}{c_{n}^{2}}\right)
-\frac{1}{c_{n}^{2}}f^{\prime}\left(  u_{c_{n}}\right)  .
\]
By the same computations as in the proof of Lemma \ref{lemma-moving-RBou},
\begin{align*}
I_{1}  &  =-\frac{2}{c_{n}^{2}}\left(  \left(  \mathcal{E}_{n}^{-}-1\right)
\left(  \mathcal{M}+1\right)  u_{c_{n}},u_{c_{n}}\right)  +\frac{1}{c_{n}^{2}%
}\left(  \left(  \mathcal{E}_{n}^{-}-1\right)  ^{2}\left(  \mathcal{M}%
+1\right)  u_{c_{n}},u_{c_{n}}\right) \\
&  \rightarrow\frac{3}{c_{0}^{2}}\left(  \left(  \mathcal{M}+1\right)
u_{c_{0}},u_{c_{0}}\right)  \text{, when }n\rightarrow\infty
\end{align*}
and
\[
I_{2}\rightarrow-\frac{2}{c_{0}^{3}}\left(  \left(  1+f^{\prime}\left(
u_{c_{0}}\right)  \right)  \partial_{c}u_{c}|_{c_{0}},u_{c_{0}}\right)
,\ \text{when }n\rightarrow\infty.
\]
Thus
\[
0=\lim_{n\rightarrow\infty}\left(  \bar{c}_{n}I_{1}+I_{2}\right)  =-\frac
{1}{c_{0}^{2}}\frac{dP}{dc}\left(  c_{0}\right)
\]
and the Lemma is proved.
\end{proof}

\section{KDV type}

Consider a solitary wave $u\left(  x,t\right)  =u_{c}\left(  x-ct\right)
\ \left(  c>0\right)  $ of the KDV type equations (\ref{kdv}). Then $u_{c}$
satisfies the equation%

\begin{equation}
\mathcal{M}u_{c}+cu_{c}-f\left(  u_{c}\right)  =0\text{.} \label{solitary-kdv}%
\end{equation}
The linearized equation is
\begin{equation}
\left(  \partial_{t}-c\partial_{x}\right)  u+\partial_{x}\left(  f^{\prime
}\left(  u_{c}\right)  u-\mathcal{M}u\right)  =0. \label{L-KDV}%
\end{equation}
and for a growing mode solution $e^{\lambda t}u\left(  x\right)  $ $\left(
\operatorname{Re}\lambda>0\right)  $, $u\left(  x\right)  $ satisfies
\begin{equation}
\left(  \lambda-c\partial_{x}\right)  u+\partial_{x}\left(  f^{\prime}\left(
u_{c}\right)  u-\mathcal{M}u\right)  =0. \label{spectrum-KDV}%
\end{equation}
We define the following dispersion operator $\mathcal{A}^{\lambda}%
:H^{m}\rightarrow L^{2}$ $\left(  \operatorname{Re}\lambda>0\right)  $
\[
\mathcal{A}^{\lambda}u=cu+\frac{c\partial_{x}}{\lambda-c\partial_{x}}\left(
f^{\prime}\left(  u_{c}\right)  u-\mathcal{M}u\right)
\]
and as before the existence of a purely growing mode is reduced to find
$\lambda>0$ such that $\mathcal{A}^{\lambda}$ has a nontrivial kernel. When
$\lambda\rightarrow0+$, $\mathcal{A}^{\lambda}$ converges to the zero-limit
operator
\begin{equation}
\mathcal{L}_{0}:=\mathcal{M}+c-f^{\prime}\left(  u_{c}\right)  .
\label{L0-KDV}%
\end{equation}

The proof of Theorem \ref{THM: main} for KDV is similar to the BBM and RBou
cases. So we only indicate some differences due to the different structure of
the operator $\mathcal{A}^{\lambda}$. To prove the essential spectrum bound%
\begin{equation}
\sigma_{\text{ess}}\left(  \mathcal{A}^{\lambda}\right)  \subset\left\{
z\ |\operatorname{Re}\lambda\geq\frac{1}{2}c\right\}  , \label{ess-KDV}%
\end{equation}
we need to establish analogues of Lemmas \ref{lemma-quadrtic-bound-lb} and
\ref{lemma-commu-d}. First, we note that, for any $u\in H^{m}\left(
\mathbf{R}\right)  ,$
\begin{align}
\operatorname{Re}\left(  -\frac{c\partial_{x}}{\lambda-c\partial_{x}%
}\mathcal{M}u,u\right)   &  =\operatorname{Re}\int\frac{-ick}{\lambda
-ick}\alpha\left(  k\right)  \left\vert \hat{\phi}\left(  k\right)
\right\vert ^{2}dk\label{inter4}\\
&  =\int\frac{\left(  ck\right)  ^{2}}{\lambda^{2}+\left(  ck\right)  ^{2}%
}\alpha\left(  k\right)  \left\vert \hat{\phi}\left(  k\right)  \right\vert
^{2}dk\geq0.\nonumber
\end{align}
So by estimates as in the proof of Lemma \ref{lemma-quadrtic-bound-lb}, for
any sequence
\[
\left\{  u_{n}\right\}  \in H^{m}\left(  \mathbf{R}\right)  ,\ \left\Vert
u_{n}\right\Vert _{2}=1,\ supp\ u_{n}\subset\left\{  x|\ \left\vert
x\right\vert \geq n\right\}  ,
\]
$\ $and any complex number $z\ $with $\operatorname{Re}z\leq\frac{1}{2}c$, we
have
\[
\operatorname{Re}\left(  \left(  \mathcal{A}^{\lambda}-z\right)  u_{n}%
,u_{n}\right)  \geq\frac{1}{4}c,
\]
when $n\ $is large enough. Since
\[
\left[  \mathcal{A}^{\lambda},\chi_{d}\right]  =\left(  1-\mathcal{E}%
^{\lambda,-}\right)  \left[  \mathcal{M},\chi_{d}\right]  +\left[
\mathcal{E}^{\lambda,-},\chi_{d}\right]  \left(  f^{\prime}\left(
u_{c}\right)  -\mathcal{M}\right)  ,
\]
the conclusion of Lemma \ref{lemma-commu-d} still holds true by the same
proof. Thus the essential spectrum bound (\ref{ess-KDV}) is obtained as
before. The non-existence of growing modes for large $\lambda$ is proved in
the following lemma.

\begin{lemma}
\label{lemma-no-eigen-infy-KDV}There exists $\Lambda>0$, such that when
$\lambda>\Lambda$, $\mathcal{A}^{\lambda}$ has no eigenvalues in $\left\{
z|\ \operatorname{Re}z\leq0\right\}  $.
\end{lemma}

\begin{proof}
Suppose otherwise, then there exists a sequence $\left\{  \lambda_{n}\right\}
\rightarrow+\infty,\ \left\{  k_{n}\right\}  \in\mathbb{C},\ $and $\left\{
u_{n}\right\}  \in$ $H^{m}\left(  \mathbf{R}\right)  $, such that
$\operatorname{Re}k_{n}\leq0$ and $\left(  \mathcal{A}^{\lambda_{n}}%
-k_{n}\right)  u_{n}=0$. Let $K>0$ be such that $\alpha\left(  k\right)  \geq
a\left\vert k\right\vert ^{m}$ when $\left\vert k\right\vert \geq K$. For any
$\delta,\varepsilon>0$, and large $n$, we have $\delta\lambda_{n}\geq K$ and
\begin{align*}
0  &  \geq\operatorname{Re}\left(  \mathcal{A}^{\lambda_{n}}u_{n},u_{n}\right)
\\
&  \geq\int\frac{\left(  ck\right)  ^{2}\alpha\left(  k\right)  }{\lambda
_{n}^{2}+\left(  ck\right)  ^{2}}\left\vert \hat{u}_{n}\left(  k\right)
\right\vert ^{2}dk+c\left\Vert u_{n}\right\Vert _{L^{2}}^{2}-\max\left\vert
f^{\prime}\left(  u_{c}\right)  \right\vert \left\Vert u_{n}\right\Vert
_{L^{2}}\left\Vert \frac{c\partial_{x}}{\lambda_{n}+c\partial_{x}}%
u_{n}\right\Vert _{L^{2}}\\
&  \geq\int\frac{\left(  ck\right)  ^{2}\alpha\left(  k\right)  }{\lambda
_{n}^{2}+\left(  ck\right)  ^{2}}\left\vert \hat{u}_{n}\left(  k\right)
\right\vert ^{2}dk+\left(  c-\varepsilon\right)  \left\Vert u_{n}\right\Vert
_{L^{2}}^{2}-\frac{\max\left\vert f^{\prime}\left(  u_{c}\right)  \right\vert
^{2}}{4\varepsilon}\int\frac{\left(  ck\right)  ^{2}}{\lambda_{n}^{2}+\left(
ck\right)  ^{2}}\left\vert \hat{u}_{n}\left(  k\right)  \right\vert ^{2}dk\\
&  \geq\int\frac{\left(  ck\right)  ^{2}\alpha\left(  k\right)  }{\lambda
_{n}^{2}+\left(  ck\right)  ^{2}}\left\vert \hat{u}_{n}\left(  k\right)
\right\vert ^{2}dk+\left(  c-\varepsilon\right)  \left\Vert u_{n}\right\Vert
_{L^{2}}^{2}-\frac{\max\left\vert f^{\prime}\left(  u_{c}\right)  \right\vert
^{2}}{4\varepsilon a\left(  \delta\lambda_{n}\right)  ^{m}}\int_{\left\vert
k\right\vert \geq\delta\lambda_{n}}\frac{\left(  ck\right)  ^{2}\alpha\left(
k\right)  }{\lambda_{n}^{2}+\left(  ck\right)  ^{2}}\left\vert \hat{u}%
_{n}\left(  k\right)  \right\vert ^{2}dk\\
\ \ \ \  &  \ \ \ \ \ -\frac{\max\left\vert f^{\prime}\left(  u_{c}\right)
\right\vert ^{2}c^{2}\delta^{2}}{4\varepsilon}\int_{\left\vert k\right\vert
\leq\delta\lambda_{n}}\left\vert \hat{u}_{n}\left(  k\right)  \right\vert
^{2}dk\\
&  \geq\left(  1-\frac{\max\left\vert f^{\prime}\left(  u_{c}\right)
\right\vert ^{2}}{4\varepsilon a\left(  \delta\lambda_{n}\right)  ^{m}%
}\right)  \int\frac{\left(  ck\right)  ^{2}\alpha\left(  k\right)  }%
{\lambda_{n}^{2}+\left(  ck\right)  ^{2}}\left\vert \hat{u}_{n}\left(
k\right)  \right\vert ^{2}dk+\left(  c-\varepsilon-\frac{\max\left\vert
f^{\prime}\left(  u_{c}\right)  \right\vert ^{2}c^{2}\delta^{2}}{4\varepsilon
}\right)  \left\Vert u_{n}\right\Vert _{L^{2}}^{2}\\
&  >0,\ \ \text{when }n\text{ is large enough,}%
\end{align*}
by choosing $\varepsilon,\delta>0$ such that
\[
\ c-\varepsilon-\frac{\max\left\vert f^{\prime}\left(  u_{c}\right)
\right\vert c^{2}\delta^{2}}{4\varepsilon}>0.
\]
This is a contradiction and the lemma is proved.
\end{proof}

The eigenvalues of $\mathcal{A}^{\lambda}$ for small $\lambda$ are also
studied by the asymptotic perturbation theory. The required analogues of
Lemmas \ref{lemma-commu-F} and \ref{lemma-quadratic bd-asymp} can be proved in
the same way. The discrete eigenvalues of $\mathcal{A}^{0}=\mathcal{L}_{0}$
are perturbed to get the eigenvalues of $\mathcal{A}^{\lambda}$ for small
$\lambda$, in the sense of Proposition \ref{Prop-asym-perturb}. The
instability criterion in Theorem \ref{THM: main} can be proved in the same
way, by deriving the following moving kernel formula: for $\lambda>0$ small
enough, let $k_{\lambda}\in\mathbf{R}$ to be the only eigenvalue of
$\mathcal{A}^{\lambda}$ near zero, then$\ $
\begin{equation}
\lim_{\lambda\rightarrow0+}\frac{k_{\lambda}}{\lambda^{2}}=-\frac{dP}%
{dc}/\left\Vert u_{cx}\right\Vert _{L^{2}}^{2}, \label{formula-kernel-KDV}%
\end{equation}
where
\begin{equation}
P\left(  c\right)  =\frac{1}{2}\left(  u_{c},u_{c}\right)  . \label{P-KDV}%
\end{equation}
We sketch the proof of (\ref{formula-kernel-KDV}) below. First, similar to
Lemma \ref{lemma-apriori}, we have the following a priori estimate:

For $\lambda>0$ small enough, if $\left(  \mathcal{A}^{\lambda}-z\right)
u=v$, $z\in\mathbb{C}$ with $\operatorname{Re}z\leq\frac{1}{2}c$ and $v\in
L^{2},$ then
\begin{equation}
\left\Vert u\right\Vert _{H^{\frac{m}{2}}}\leq C\left(  \left\Vert
u\right\Vert _{L_{e}^{2}}+\left\Vert v\right\Vert _{L^{2}}\right)  ,
\label{estimate-KDV}%
\end{equation}
for a constant $C$ independent of $\lambda$. To prove (\ref{estimate-KDV}), we
note that for any $\varepsilon>0$
\[
\operatorname{Re}\left(  \mathcal{A}^{\lambda}u,u\right)  -\frac{1}%
{2}c\left\Vert u\right\Vert _{L^{2}}^{2}\leq\left\Vert u\right\Vert _{L^{2}%
}\left\Vert v\right\Vert _{L^{2}}\leq\varepsilon\left\Vert u\right\Vert
_{L^{2}}^{2}+\frac{1}{4\varepsilon}\left\Vert v\right\Vert _{L^{2}}^{2}%
\]
and for any $\delta>0$, when $\lambda\leq cK$,
\begin{align*}
\operatorname{Re}\left(  \mathcal{A}^{\lambda}u,u\right)   &  \geq\int
\frac{\left(  ck\right)  ^{2}\alpha\left(  k\right)  }{\lambda^{2}+\left(
ck\right)  ^{2}}\left\vert \hat{u}\left(  k\right)  \right\vert ^{2}%
dk+c\left\Vert u\right\Vert _{L^{2}}^{2}-\left\Vert u\right\Vert _{L_{e}^{2}%
}\left\Vert u_{n}\right\Vert _{L^{2}}\\
&  \geq\frac{a}{2}\int_{\left\vert k\right\vert \geq K}\left\vert k\right\vert
^{m}\left\vert \hat{u}\left(  k\right)  \right\vert ^{2}dk+c\left\Vert
u\right\Vert _{L^{2}}^{2}-\varepsilon\left\Vert u\right\Vert _{L^{2}}%
^{2}-\frac{1}{4\varepsilon}\left\Vert u\right\Vert _{L_{e}^{2}}^{2}\\
&  \geq\min\left\{  \frac{a}{2},\frac{\delta}{K^{m}}\right\}  \int\left\vert
k\right\vert ^{m}\left\vert \hat{u}\left(  k\right)  \right\vert
^{2}dk+\left(  c-\delta-\varepsilon\right)  \left\Vert u\right\Vert _{L^{2}%
}^{2}-\frac{1}{4\varepsilon}\left\Vert u\right\Vert _{L_{e}^{2}}^{2}.
\end{align*}
Thus by choosing $\delta,\varepsilon$ to be small, we get the estimate
(\ref{estimate-KDV})$.$

To prove (\ref{formula-kernel-KDV}), we follow the same procedures as in the
BBM and RBou cases. Let $u_{\lambda}\in H^{m}\left(  \mathbf{R}\right)  $ be
the solution of $\left(  \mathcal{A}^{\lambda}-k_{\lambda}\right)  u_{\lambda
}=0$ with $k_{\lambda}\in\mathbf{R}$ and $\lim_{\lambda\rightarrow
0+}k_{\lambda}=0$. We normalize $u_{\lambda}$ by setting $\left\Vert
u_{\lambda}\right\Vert _{L_{e}^{2}}=1$. Then by (\ref{estimate-KDV}), we have
$\left\Vert u_{\lambda}\right\Vert _{H^{\frac{m}{2}}}\leq C$ and as before,
after a renormalization $u_{\lambda}\rightarrow u_{0}=u_{cx}$ in $H^{\frac
{m}{2}}$. We have $\lim_{\lambda\rightarrow0+}\frac{k_{\lambda}}{\lambda}=0$,
since
\begin{align*}
\frac{k_{\lambda}}{\lambda}\left(  u_{\lambda},u_{cx}\right)   &  =\left(
\frac{\mathcal{A}^{\lambda}-\mathcal{A}^{0}}{\lambda}u_{\lambda}%
,u_{cx}\right)  =\left(  \frac{1}{\lambda-\mathcal{D}}\left(  f^{\prime
}\left(  u_{c}\right)  -\mathcal{M}\right)  u_{\lambda},u_{cx}\right) \\
&  =\frac{1}{c}\left(  \left(  f^{\prime}\left(  u_{c}\right)  -\mathcal{M}%
\right)  u_{\lambda},\frac{\mathcal{D}}{\lambda-\mathcal{D}}u_{c}\right) \\
&  \rightarrow-\frac{1}{c}\left(  \left(  f^{\prime}\left(  u_{c}\right)
-\mathcal{M}\right)  u_{cx},u_{c}\right)  =-\left(  u_{cx},u_{c}\right)
=0\text{.}%
\end{align*}
Similarly as before, we can show that $u_{\lambda}=\bar{c}_{\lambda}%
u_{cx}+\lambda\bar{v}_{\lambda}$, with $\bar{c}_{\lambda}\rightarrow1$,
$\bar{v}_{\lambda}\rightarrow-\partial_{c}u_{c}$ in $H^{\frac{m}{2}},\ $when
$\lambda\rightarrow0+$. In the proof, we use the facts that
\begin{align*}
w_{\lambda}\left(  x\right)   &  =\frac{\mathcal{A}^{\lambda}-\mathcal{A}^{0}%
}{\lambda}u_{cx}=\frac{1}{\lambda-\mathcal{D}}\left(  f^{\prime}\left(
u_{c}\right)  -\mathcal{M}\right)  u_{cx}\\
&  =\frac{\mathcal{D}}{\lambda-\mathcal{D}}u_{c}\rightarrow-u_{c}\text{, when
}\lambda\rightarrow0+,
\end{align*}
and $\mathcal{L}_{0}\partial_{c}u_{c}=-u_{c}$. Now
\[
\frac{k_{\lambda}}{\lambda^{2}}\left(  u_{\lambda},u_{cx}\right)  =\bar
{c}_{\lambda}\left(  \frac{\mathcal{A}^{\lambda}-\mathcal{A}^{0}}{\lambda^{2}%
}u_{cx},u_{cx}\right)  +\left(  \frac{\mathcal{A}^{\lambda}-\mathcal{A}^{0}%
}{\lambda}\bar{v}_{\lambda},u_{cx}\right)  =\bar{c}_{\lambda}I_{1}+I_{2}%
\]
and
\begin{align*}
I_{1}  &  =\left(  \frac{w_{\lambda}\left(  x\right)  }{\lambda}%
,u_{cx}\right)  =\left(  \frac{1}{\left(  \lambda-\mathcal{D}\right)  \lambda
}\left(  f^{\prime}\left(  u_{c}\right)  -\mathcal{M}\right)  u_{cx}%
,u_{cx}\right) \\
&  =-\frac{1}{c}\left(  \frac{1}{\lambda-\mathcal{D}}\left(  f^{\prime}\left(
u_{c}\right)  -\mathcal{M}\right)  u_{cx},u_{c}\right)  +\frac{1}{c\lambda
}\left(  \left(  f^{\prime}\left(  u_{c}\right)  -\mathcal{M}\right)
u_{cx},u_{c}\right) \\
&  =-\frac{1}{c}\left(  \left(  \mathcal{E}^{\lambda,-}-1\right)  u_{c}%
,u_{c}\right)  \rightarrow\frac{1}{c}\left(  u_{c},u_{c}\right)  ,
\end{align*}%
\begin{align*}
I_{2}  &  =\left(  \frac{\mathcal{A}^{\lambda}-\mathcal{A}^{0}}{\lambda}%
\bar{v}_{\lambda},u_{cx}\right)  =-\frac{1}{c}\left(  \frac{\mathcal{D}%
}{\lambda-\mathcal{D}}\left(  f^{\prime}\left(  u_{c}\right)  -\mathcal{M}%
\right)  \bar{v}_{\lambda},u_{c}\right) \\
&  \rightarrow-\frac{1}{c}\left(  \left(  f^{\prime}\left(  u_{c}\right)
-\mathcal{M}\right)  \partial_{c}u_{c},u_{c}\right)  =-\frac{1}{c}\left(
c\partial_{c}u_{c}+u_{c},u_{c}\right)  ,
\end{align*}
so
\[
\lim_{\lambda\rightarrow0+}\frac{k_{\lambda}}{\lambda^{2}}=\lim_{\lambda
\rightarrow0+}\frac{\bar{c}_{\lambda}I_{1}+I_{2}}{\left(  u_{\lambda}%
,u_{cx}\right)  }=-\left(  \partial_{c}u_{c},u_{c}\right)  /\left(
u_{cx},u_{cx}\right)  =-\frac{dP}{dc}/\left\Vert u_{cx}\right\Vert _{L^{2}%
}^{2}\text{.}%
\]

\section{Discussions}

\begin{center}
\textit{(a) About the spectral assumption for }$\mathcal{L}_{0}$
\end{center}

When $\mathcal{M}=-\frac{d}{dx^{2}}$, the assumption (\ref{assum-kernel}) that
$\ker\left(  \mathcal{L}_{0}\right)  =\left\{  u_{cx}\right\}  $ is true
because the second order ODE $\mathcal{L}_{0}\psi=0\ $has two solutions which
decay and grow at infinity respectively, and thus $u_{cx}$ is the only
decaying solution. Moreover, the solitary waves in such case can be shown to
be positive and single-humped. Thus by the Sturm-Liouville theory for second
order ODE operators, $n^{-}\left(  \mathcal{L}_{0}\right)  =1$ since $u_{cx}$
has exactly one zero. The proof of (\ref{assum-kernel})$\ $for nonlocal
dispersive operator $\mathcal{M}$ is much more delicate. In (\cite{albert92},
\cite{albert-et-91}), (\ref{assum-kernel}) is proved for solitary waves of
some KDV type equations, such as the intermediate long-wave equation
(\cite{ILW-eqn}) with
\[
f\left(  u\right)  =u^{2}\ \text{and\ }\alpha\left(  k\right)  =k\coth\left(
kH\right)  -H^{-1}.
\]
The assumption (\ref{assum-kernel}) is related to the bifurcation of solitary
waves, in the sense that $\ker\mathcal{L}_{0}=\left\{  u_{cx}\right\}  $
implies the nonexistence of secondary bifurcations at $c$, that is, the
solitary wave branch $u_{c}\left(  x\right)  $ is locally unique. Even in
cases of multiple branches of solitary waves, (\ref{assum-kernel}) is still
valid in each branch. We note that $\ker\mathcal{L}_{0}$ also monitors the
changes of $n^{-}\left(  \mathcal{L}_{0}\right)  $ when $c$ is changed. For
example, when (\ref{assum-kernel}) is valid in a certain range of $c$,
$n^{-}\left(  \mathcal{L}_{0}\right)  $ must remain unchanged in this range.
Since otherwise, by continuation there is a crossing of eigenvalues through
origin at some $c$, which increase the dimension of $\ker\mathcal{L}_{0}.$This
observation has been used in some problems (\cite{albert-et-91},
\cite{lin-soli-water}) to get $n^{-}\left(  \mathcal{L}_{0}\right)  $ for
large waves from small waves for which $n^{-}\left(  \mathcal{L}_{0}\right)  $
is computable. At secondary bifurcation and turning points, the increase of
$\ker\left(  \mathcal{L}_{0}\right)  $ signals the increase or decrease of
$n^{-}\left(  \mathcal{L}_{0}\right)  $ when these transition points are
crossed$.$One such example is the solitary waves for full water wave problem
(\cite{lin-soli-water}), for which the infinitely many turning points makes
$n^{-}\left(  \mathcal{L}_{0}\right)  $ to increase without bound by a result
of Plotnikov.

The assumption (\ref{assum-kernel}) is also required in all existing proof of
orbital stability (\cite{gss87}, \cite{bss87}, \cite{wein-87}).

\begin{center}
\textit{(b) The sign-changing symbol}
\end{center}

We assume $\alpha\left(  k\right)  \geq0$ in our proof of Theorem
\ref{THM: main}. The proof can be easily modified to treat sign-changing
symbols. Let $-\gamma=\inf\alpha\left(  k\right)  <0$. Consider solitary wave
solutions of KDV, BBM, and RBou type equations with
\begin{equation}
c>\gamma,\ 1-\frac{1}{c}>\gamma\ \ \ \text{and\ \ }1-\frac{1}{c^{2}}%
>\gamma\label{condition-c}%
\end{equation}
respectively. The condition (\ref{condition-c}) on $c$ is to ensure that the
essential spectrum of $\mathcal{L}_{0}$ lies in the positive axis, which is
required to get decaying solitary waves, such as in \cite{amt2} and
\cite{abr-benjamin} for fifth order KDV and Benjamin equations with
$\alpha\left(  k\right)  =-k^{2}+\delta k^{4}$ and $-\left\vert k\right\vert
+\delta k^{2}$ respectively. Denote $\mathcal{\tilde{M}}$ to be the multiplier
operator with the nonnegative symbol $\tilde{\alpha}\left(  k\right)
=\alpha\left(  k\right)  +\gamma$. The proof of Theorem \ref{THM: main}
remains unchanged, by replacing $\mathcal{M}$ with $\mathcal{\tilde{M}}%
-\gamma$ and using the nonnegative symbol $\tilde{\alpha}\left(  k\right)  $
in estimates. The same estimates still go through because of the condition
(\ref{condition-c}). For sign-changing symbols, the solitary waves might be
highly oscillatory in some parameter range (\cite{amt2}, \cite{abr-benjamin}).
It is conceivable that such oscillatory waves are energy saddle with
$n^{-}\left(  \mathcal{L}_{0}\right)  \geq2$, whose stability can not be
studied by the traditional energy minimizer idea. Theorem \ref{THM: main}
gives a sufficient condition for instability in such cases.

\begin{center}
\textit{(c) Comparisons with the Evans function method}
\end{center}

In \cite{pw92-evans}, Pego and Weinstein use the Evans function technique to
obtain the instability criterion $dP/dc<0$ for the case $\mathcal{M}=-\frac
{d}{dx^{2}}$. In their paper, the eigenvalue problems (\ref{spectral-BBM}),
(\ref{spectral-Rbou}) and (\ref{spectrum-KDV}) are written as a first order
system in $x$, depending on the parameter $\lambda$. The Evans function
$D\left(  \lambda\right)  $ is a Wronskian-like function whose zeros in the
right half-plane correspond to unstable eigenvalues, and it measures the
intersection of subspaces of solutions exponentially decaying at $+\infty$ and
$-\infty$. This method was first introduced by J . W. Evans in a series papers
including \cite{evans2} and further studied in \cite{agj-evans}. In
\cite{pw92-evans}$,$it is shown that $D\left(  \lambda\right)  >0$ when
$\lambda>0$ is big enough, $D\left(  0\right)  =D^{\prime}\left(  0\right)
=0$ and
\begin{equation}
D^{\prime\prime}\left(  0\right)  =sgn\ dP/dc. \label{D''}%
\end{equation}
If $dP/dc<0$, then $D\left(  \lambda\right)  <0$ and a continuation argument
yield the vanishing of $D\left(  \lambda\right)  $ at some $\lambda>0$, which
establishes a growing mode. A similar formula as (\ref{D''}) is derived
in\ \cite{bridges-derks}, for problems which can be written in a
multi-symmpletic form. However, there are several restrictions of the Evans
function method: 1) Only the differential operators, that is, with polynomial
symbols, can be treated, since the eigenvalue problems need to be written as a
first order system. 2) The solitary waves must have the exponential decay.
Moreover, certain assumptions for eigenvalues of the asymptotic systems are
required in constructing the Evans function (\cite[(0.6), (0.7)]{pw92-evans}).
Such assumptions need to be checked case by case, and their relations to the
properties of solitary waves are not very clear. By comparison, our approach
apply to very general dispersive operators, in particular, nonlocal operators.
We impose no additional assumptions on the solitary waves. For example, we
allow slowly decaying, highly oscillatory or non-symmetric solitary waves. Our
only assumption (\ref{assum-kernel}) is closed related to the bifurcation of
solitary waves, and it appears to be rather natural in the stability theory.
Moreover, the Evans function method can only be used for the one-dimensional
problems, since otherwise the first order system can not be written. Our
approach has no such restriction and might be useful in the multi-dimensional setting.

Lastly, we note that in Theorem \ref{THM: main}, the instability is determined
by both the sign of $dP/dc$ and the oddness of $n^{-}\left(  \mathcal{L}%
_{0}\right)  $. The later information seems to not appear in the Evans
function method (\cite{pw92-evans}, \cite{bridges-derks}). When $\mathcal{M}$
is a differential operator and $n^{-}\left(  \mathcal{L}_{0}\right)  $ is
even, suppose the Evans function can be constructed and the formula
(\ref{D''}) is shown, then the instability criterion would be still $dP/dc<0$,
which is different from the instability criterion $dP/dc>0$ by Theorem
\ref{THM: main}. It would be interesting to clarify this issue. On possible
such example is the oscillatory solitary waves (\cite{amt2}) of the fifth
order KDV equation.

\begin{center}
\textit{(d) Some future problems}
\end{center}

There are several open issues from our study.

(i) When the instability conditions in Theorem \ref{THM: main} are not
satisfied, the stability of the solitary waves is unknown, except for the case
when $n^{-}\left(  \mathcal{L}_{0}\right)  =1$ in the KDV and BBM case. Such
solitary waves are energy saddles of an even negative index, whose stability
is very subtle and not resolved even for the finite dimensional Hamiltonian
systems. One might need to look for the oscillatory growing modes in such cases.

(ii) The nonlinear stability of solitary waves of energy saddles type is
entirely open. This problem is important because of its direct relevance to
the full water wave problem. Theorem \ref{thm-transition} might be useful to
study spectral stability as a first step. To apply it, one need to understand
when the oscillatory instability can be excluded, which is related to (i).

(iii) Can we get nonlinear instability from linear instability, in the $L^{2}$
norm? This problem is open, even in the KDV and BBM cases where the nonlinear
instability in the energy norm has been proved (\cite{bss87}, \cite{gss90}).
This problem is also relevant to full water waves and other problems for which
the blow-up issue is concerned. The $L^{2}$ instability results could be used
to distinguish the large scale instability of basic waves from the local
blow-up instability due to the structure of the models.

\section{Appendix}

In this Appendix, we describe a different approach than \cite{bss87} and
\cite{ss89} to get nonlinear instability for some dispersive wave models. In
\cite{bss87} and \cite{ss89}, the Liapunov functional method of \cite{gss87}
is extended to get nonlinear instability of solitary waves of KDV and BBM type
equations, under the assumptions $dP/dc<0$ and (\ref{kernel-1}). For the KDV
case, the Liapunov functional constructed in \cite{bss87} becomes
\begin{equation}
A(t)=\int Y\left(  x-x\left(  t\right)  \right)  u\left(  x,t\right)  dx,
\label{defn-A}%
\end{equation}
where $Y\left(  x\right)  =\int_{-\infty}^{x}$ $y\left(  z\right)  \ dz$ and
$y\left(  x\right)  $ is an energy decreasing direction under the constraint
of the constant momentum $Q\left(  u\right)  $. By using the fact that the
solitary wave considered is an (constrainted) energy saddle with negative
index one, it can be shown (\cite{gss87}) that $A^{\prime}\left(  t\right)
\geq\delta>0$ in the orbital neighborhood of the solitary wave. The nonlinear
instability would follow immediately if $A\left(  t\right)  $ is bounded, as
considered in the abstract setting of \cite{gss87}. However, $A\left(
t\right)  $ defined by (\ref{defn-A}) is\ not bounded because the function
$Y\left(  x\right)  $ is not in $L^{2}$ if $\int y\ dx\neq0$. To overcome this
issue, in \cite{bss87} it is shown that $A\left(  t\right)  \leq C\left(
1+t^{\eta}\right)  $ with some $\eta<1$, then the nonlinear instability still
follows. Such an estimate is obtained by showing that the maximum of the
anti-derivative of $u\left(  x,t\right)  $ has a sublinear growth. The same
approach is used in \cite{ss89}, \cite{liu93} and \cite{debd96} (for KP
equations), and the sublinear estimates are sometimes highly nontrivial to
prove. Below, we show that such an estimate can be avoided by using another
approach, which was first introduced in \cite{lin-bubble} for a
Schr\"{o}dinger type problem.

The idea in \cite{lin-bubble} is to make a small correction to the (energy)
decreasing direction $y\left(  x\right)  $ used in constructing the Liapunov
functional $A\left(  t\right)  $. The new direction, still decreasing, has the
additional property that its integral over $\mathbf{R}$ is zero. Then the new
anti-derivative $Y\left(  x\right)  \in L^{2}$ and thus $A\left(  t\right)  $
is bounded which implies nonlinear instability. The correction is through the
following lemma, which is a generalization of \cite[Lemma 5.2]{lin-bubble}$.$

\begin{lemma}
For any $r(x)\neq0\in L^{2}(\mathbf{R}),\ c\in\mathbf{R}$ and ${}m\geq1$,
there exists a sequence $\{y_{n}\}$ in $H^{m}(\mathbf{R})$ such that
\[
(1+|x|)y_{n}(x)\in L^{1}(\mathbf{R}{}),\ \ \int y_{n}\left(  x\right)
\ dx=c,
\]
$y_{n}\rightarrow0$ in $H^{\frac{m}{2}}(\mathbf{R})$ and $(y_{n},r)=0.$
\end{lemma}

\begin{proof}
We choose $\varphi(x)\in C_{0}^{\infty}\left(  \mathbf{R}\right)  $ such that
$\int\varphi(x)dx=c$. We claim that: there exists $\psi(x)\in C_{0}^{\infty
}(\mathbf{R})$ such that $\left(  \psi_{x},r\right)  \neq0.$Suppose otherwise,
for any $\psi(x)\in C_{0}^{\infty}({\mathbf{R}}^{1}),\ $we have $\left(
\psi_{x},r\right)  =0.$Then $r_{x}=0$ in the distribution sense and thus
$r\equiv$ constant. But $r\in L^{2}$, so $r=0$, which is a contradiction.
Define
\[
y_{n}=\frac{1}{n}\varphi(\frac{x}{n})-a_{n}\psi_{x}(x),
\]
with%
\[
a_{n}=\frac{\int\frac{1}{n}\varphi(\frac{x}{n})r(x)dx}{\left(  \psi
_{x},r\right)  }.
\]
Then $(y_{n},r)=0$ and
\[
|a_{n}|\leq\frac{||\frac{1}{n}\varphi(\frac{1}{n}x)||_{2}||r||_{2}}{\left\vert
\left(  \psi_{x},r\right)  \right\vert }=O\left(  \frac{1}{\sqrt{n}}\right)
\rightarrow0,
\]
when $n\rightarrow\infty$. Let $\varphi_{n}\left(  x\right)  =\varphi(\frac
{x}{n})$, then
\begin{align*}
\left\Vert \frac{1}{n}\varphi(\frac{x}{n})\right\Vert _{H^{\frac{m}{2}}}^{2}
&  =\frac{1}{n}\left\Vert \varphi\right\Vert _{L^{2}}^{2}+\frac{1}{n^{2}%
}\left\Vert \left\vert D\right\vert ^{\frac{m}{2}}\varphi_{n}\right\Vert
_{L^{2}}^{2}=\frac{1}{n}\left\Vert \varphi\right\Vert _{L^{2}}^{2}+\frac
{1}{n^{m+1}}\left\Vert \left\vert D\right\vert ^{\frac{m}{2}}\varphi
\right\Vert _{L^{2}}^{2}\\
&  \rightarrow0\text{, when }n\rightarrow\infty\text{,}%
\end{align*}
where in the above we use the scaling formula%
\[
\left\vert D\right\vert ^{\frac{m}{2}}\varphi_{n}\left(  x\right)  =\frac
{1}{n^{\frac{m}{2}}}\left(  \left\vert D\right\vert ^{\frac{m}{2}}%
\varphi\right)  \left(  \frac{x}{n}\right)
\]
as in the proof of Lemma \ref{lemma-commu-d}. Therefore, $y_{n}\rightarrow0$
in $H^{\frac{m}{2}}(\mathbf{R})$, $(1+|x|)y_{n}\in L^{1}$ and
\[
\int y_{n}(x)dx=\int\frac{1}{n}\varphi(\frac{x}{n})dx=\int\varphi(x)dx=c.
\]
The lemma is proved.
\end{proof}

We start with an (constrainted) energy decreasing direction $y\left(
x\right)  $ with $(1+|x|)y(x)\in L^{1}$, that is,
\[
\left(  \mathcal{H}y,y\right)  <0\text{ and }\left(  y,Q^{\prime}\left(
u_{c}\right)  \right)  =0,
\]
where $\mathcal{H}$ is the second order variation of the argumented energy
functional, for which the solitary wave is a critical point. Let $H^{\frac
{m}{2}}$ to be the energy space, that is, $m$ is the power of the operator
$\mathcal{H}$. Choosing $c=\int y\ dx,\ r=Q^{\prime}\left(  u_{c}\right)  $ in
the above lemma, we get a sequence $\left\{  y_{n}\right\}  \in$ $H^{m}$ with
the properties listed in the lemma. Defining $\tilde{y}_{n}=y-y_{n}$, then we
have
\[
(1+|x|)\tilde{y}_{n}(x)\in L^{1},\ \ \left(  \tilde{y}_{n},P^{\prime}\left(
u_{c}\right)  \right)  =0,\ \int\tilde{y}_{n}\ dx=0\
\]
and $\left(  \mathcal{H}\tilde{y}_{n},\tilde{y}_{n}\right)  <0$ when $n$ is
big enough$.$ Thus for large $n$, the function $\tilde{y}_{n}$ is a new
(constrainted) energy decreasing direction with zero integral. The Liapunov
functional $A\left(  t\right)  $ is defined as in (\ref{defn-A}) by using this
new direction $\tilde{y}_{n}.\ $By \cite[p. 409]{bss87}, $Y\left(  x\right)
=\int_{-\infty}^{x}$ $\tilde{y}_{n}\left(  z\right)  \ dz$ is in $L^{2}$, thus
$A\left(  t\right)  $ is bounded and the nonlinear instability results. Above
approach has the following physical interpretation: if a solitary wave is not
an energy minimizer under the constraint of constant momentum, neither is it
even under the additional constraint of constant mass. This rather general
idea could be useful in proving nonlinear instability of (constarinted) energy
saddles with index one, for other similar problems.

\begin{center}
{\Large Acknowledgement}
\end{center}

This work is supported partly by the NSF grants DMS-0505460 and DMS-0707397.
The author thanks Yue Liu for helpful discussions and Lixin Yan for
discussions on the proof of Lemma \ref{lemma-commu-d}.

\end{document}